\documentclass[12pt ,reqno]{amsart}
\usepackage{amssymb,amscd,verbatim,amsthm}
\usepackage[dvips]{color}

\newfont{\sans}{cmss10}
\newfont{\sansm}{cmss8}
\newfont{\biig}{cmbx12}

\newcommand\ovl{\overline}

\newcommand \fk[1]{{{\mathfrak #1}}}
\newcommand \C[1]{{\mathcal #1}}
\newcommand \ov[1]{{\overline {#1}}}

\newcommand \bb[1]{{\mathbb #1}}
\newcommand \wti[1]{{\widetilde #1}}

\newcommand \ie{{\it i.e.~ }}

\newcommand\CSA{{\it Cartan subalgebra}}
\newcommand\CSG{{\it Cartan subgroup}}

\newcommand\CSGs{{\it Cartan subgroups}}
\newcommand\DS{{\it Discrete Series }}
\newcommand\LDS{{\it Limit of Discrete Series}}

\newcommand\EP{{\it Euler-Poincar\'e}}
\newcommand\als{{\it almost semisimple}}

\newcommand\oA{{{^0}\!A}}

\newcommand\oG{{{^0}\!G}}
\newcommand\obG{{{^0}\bG}}

\newcommand\G{{\mathbf G}}
\newcommand\bG{{\mathbb G}}
\newcommand\bT{{\mathbb T}}
\newcommand\Hb{{\mathbf H}}

\newcommand\T{{\mathbf T}}

\newcommand \ep{{\epsilon}}
\newcommand \la{{\lambda}}
\newcommand \om{{\omega}}
\newcommand\sig{{\sigma}}
\newcommand\al{{\alpha}}

\newcommand\La{{\Lambda}}
\newcommand\cha{{\check\al}}

\newcommand\Gsn{{\mathbf G }}
\newcommand\Gti{{\widetilde{ G }}}
\newcommand\bGti{{\widetilde{\bG}}}

\newcommand\wwH{\widetilde{\widetilde{H}}}

\newcommand{\flef}{f_{\C L}}

\newcommand\Gz{{G^0}}
\newcommand\Gsr{{\wti G_{sr}}}

\newcommand\one{{{1\!\!1}}}

\newtheorem*{corollary}{Corollary}
\newtheorem*{definition}{Definition}
\newtheorem*{lemma}{Lemma}

\newtheorem*{proposition}{Proposition}

\newtheorem*{theorem}{Theorem}

\numberwithin{equation}{subsection}

\newcommand{\bk}{{\mathbf k}}

\newcommand{\bn}{{\mathbf n}}

\newcommand{\bU}{{\mathbb U}}
\newcommand{\bA}{{\mathbb A}}
\newcommand{\bC}{{\mathbb C}}
\newcommand{\bF}{{\mathbb F}}
\newcommand{\bN}{{\mathbb N}}
\newcommand{\bQ}{{\mathbb Q}}
\newcommand{\bR}{{\mathbb R}}
\newcommand{\bZ}{{\mathbb Z}}
\newcommand{\bK}{{\mathbf K}}
\newcommand{\bM}{{\mathbb M}}
\newcommand{\bP}{{\mathbb P}}

\newcommand{\CP}{\C P}

\newcommand\kbar{\ovl{\bk}}
\newcommand\Ad{{\operatorname{Ad}}}
\newcommand\Hom{{\operatorname{Hom}}}

\newcommand\im{{\operatorname{im}}}
\newcommand\diag{{\operatorname{diag}}}
\newcommand\tr{{\operatorname{tr}}}
\newcommand\Ext{{\operatorname{Ext}}}
\newcommand\vol{{\operatorname{vol}}}
\newcommand\Ho{{\operatorname{H}}}
\newcommand\rank{\operatorname{rk}}
\newcommand\Ind{\operatorname{Ind}}

\newcommand\Ar{{\cite{Ar}}}
\newcommand\BLS{\cite{BLS}}
\newcommand\Cf{{\cite{Cf}}}
\newcommand\Helg{{\cite{Helg}}}
\newcommand\KnV{\cite{KnV} }
\newcommand\KRS{\cite{KRS}}
\newcommand\LABa{\cite{LAB1}}

\newcommand\KO{\cite{KO}}
\newcommand\RS{\cite{RS}}
\newcommand \Se{\cite{Se}}
\newcommand\STa{\cite{ST1}}

\newcommand \Wa{\cite{Wa}}

\begin{document}

\title[Cuspidal representations]
{\bf \large Cuspidal representations of reductive groups}
\date{\today}
\footnotetext{$^\dagger$Both authors are partially supported by NSF grant
DMS-0070561}
\author[Dan Barbasch and Birgit Speh]{ Dan Barbasch and Birgit Speh$^\dagger$}

\address{Dan Barbasch\\
Department of Mathematics\\
310 Malott Hall\\
Cornell University\\
Ithaca, NY 14853-4201\\
U.S.A.} \email{barbasch@math.cornell.edu}
\address{Birgit Speh \\Department of Mathematics\\
310 Malott Hall\\
Cornell University\\
Ithaca, NY 14853-4201\\
U.S.A.} \email{speh@math.cornell.edu}

\maketitle \setcounter{section}{0}

\begin{abstract}
The goal of this paper is to prove the
existence of cuspidal automorphic representations of a reductive group $\bG$
which are invariant under an (outer) automorphism $\tau$ of finite order.
In particular we focus on the well known examples are $\bG=GL(n)$ with
$\tau(x):= \ ^t x^{-1}$ and in the even case the inner twist with
fixed points $Sp(2n).$ 
Our main tool is the twisted Arthur trace formula, and a local
analysis of orbital integrals and  Lefschetz numbers of representations.
\end{abstract}

\bigskip
\section*{\bf Outline}

\begin{description}
\item[I] Introduction

\item[II] Assumptions and Notation

\item[III] Twisted conjugacy classes

\item[IV] Orbital Integrals

\item[V] Finite dimensional representations

\item[VI] Lefschetz numbers

\item[VII] Lefschetz functions in the real case

\item[VIII] Lefschetz functions in the p-adic case

\item[IX] The twisted trace formula

\item[X]  A simplification of theorem \ref{8.4}

\item[XI] The main theorems

\end{description}

\section{\bf Introduction}\label{0}
{
In classical analytic number theory automorphic functions  are
holomorphic functions on the upper half plane  $\C H =
SL(2,\bR)/SO(2)$ with a prescribed transformation rule under a
subgroup of finite index $\Gamma$ of $SL(2,\bZ)$. Automorphic
functions lift to square 
integrable functions on $L^2(\Gamma \backslash G)$ with respect to an
invariant measure, and under the right action of $SL(2,\bR)$ theses
functions generate a subspace of $L^2(\Gamma \backslash G)$ which
decomposes into a direct sum of irreducible automorphic cuspidal  
representations of $SL(2,\bR)$.

 A well known generalization is as follows. Let $G$ be a semisimple
 non compact Lie group and $\Gamma$ a  
discrete subgroup of finite covolume  with respect to a some right
invariant measure $dg$, and let $L^2(\Gamma \backslash G)$ be the space
of square integrable functions with respect to $dg$. { When
$\Gamma\backslash G$ is compact, $L^2(\Gamma\backslash G)$ decomposes
into a direct sum of irreducible representations occuring with finite
mulitplicity by results of Gelfand and Piatetsky-Shapiro.  
When $\Gamma \backslash G$ is not compact, an irreducible
(necessarily unitary)   representation $\Pi$ is said to be
automorphic with respect to $\Gamma$ if it occurs discretely with
finite multiplicity in $L^2(\Gamma \backslash G).$ These
representations are also referred to as the \textit{discrete spectrum}
of $L^2(\Gamma \backslash G)$. The discrete spectrum contains a  $G-$invariant 
subspace denoted $L^2_0(\Gamma \backslash G )$. Functions in
$L^2_0(\Gamma \backslash G )$ are called cuspidal, and are
characterized by the property that they
decay very rapidly at the cusps of $\Gamma \backslash G$.} For $G =
SL(2,\bR)$ the representation generated by a given automorphic function
is in  $L^2_0(\Gamma \backslash G )$.  The complement of the cuspidal
spectrum in the discrete spectrum is called the residual spectrum. 
One of the major unsolved
problems in the theory of automorphic forms is to determine  
the multiplicities of the irreducible  representations occuring 
in  $L^2_0(\Gamma \backslash G ).$ }
 
The techniques used to show that
certain representations $\Pi$ occur with nonzero multiplicity
either exploit the connection with the geometry of
the corresponding locally symmetric space (see for example \cite{RS}
and references therein),   or use the Arthur trace formula (see 
\cite{BLS} and the references therein).

\medskip
  For  the construction of  representations in the
  residual spectrum, it is also very important to know the existence 
 of cuspidal representations invariant under an automorphism $\tau$ of
$\bG.$ For example $GL(n,\bR)$ is the Levi component of a parabolic
 subgroup of the split real form of $\bG=SO(2n)$ or $SO(2n+1),$ as
 well as $\G=Sp(2n).$ 
Let $\pi$ be a cuspidal automorphic representation of $GL(n,\bR).$
The residual spectrum of 
the Eisenstein series associated to the induced modules
$\Ind_{GL(n)}^G[\chi_s\otimes\pi],$ where $\chi_s$ is a character of
$GL(n,\bR)$, is tied to the nature of the poles of the
$L-$functions $L(s,\pi,S^2\bC^n)$ and $L(s,\pi,\bigwedge^2\bC^n)$.
With the appropriate normalization, the product of these two functions
has a simple pole at $s=1$ precisely when $\pi$ is invariant under the
outer automorphism of $GL(n).$  A  detailed discussion of results
and conjectures about the poles of these $L-$functions can be found in
\cite{BG} and \cite{BF}.

\subsection{} The main goal of this paper is to prove the existence of
cuspidal automorphic representations of reductive groups G. We are in
particular interested in those representations with integral nonsingular
infinitesimal character which are also invariant
under an automorphism of the group $G.$ Our main tool is the Arthur
trace formula together with local harmonic analysis. 
All the local results hold for arbitrary
fields. But the global techniques mostly apply to the case of a totally real
number field $\bK.$ In the interest of clarity in the global situation we
present the case of $\bK=\bQ$ only. 

So let $\bG/\bQ$ be a connected
reductive algebraic group so that $\bG(\bR) $ is noncompact.  Let
$\tau:\bG \rightarrow \bG$ be a $\bQ-$rational automorphism of
finite order. The automorphism acts on the cuspidal automorphic functions on
$\bG(\bA)$. { If $F$ is a finite dimensional
  representation of $\bG(\bR)\rtimes\{1,\tau\},$ then  $\tr F(\tau)$ is
  well defined. Note that if $\tr F(\tau)\ne 0,$ then the
  restriction of $F$ to $\bG(\bR)$ must be irreducible.
} 

The main result is the following theorem.

\medskip
\begin{theorem}[theorem \ref{t:10.1.2}] Let $\mathbb G$ be a connected
  reductive linear algebraic group
defined over   $\bQ$, {and assume that $G(\bR)$ has no compact
  factors}. { Let $F$  be a finite dimensional irreducible
representation of $\mathbb G({\bR})\ltimes\{1,\tau\},$  and assume
that the centralizer of $\tau$ in $\fk g (\bR)$ is of  equal rank.} If
$\tr F(\tau)\ne 0$, then there exists a cuspidal automorphic
representation $\pi_\bA$  of ${\mathbb G}({\bf A})$ stable under
$\tau$, with the same infinitesimal character as $F.$

In addition 
\begin{equation}
  \label{eq:0.1}
  \Ho^* (\fk g(\bR), K_\infty, \pi_\bA\otimes F )\ne 0.
\end{equation}

\end{theorem}
 
\medskip
{If $\tau $ is an involution, $\tr F(\tau)$ is
computed in \cite{RS} (see \ref{4.3.3} for a more uniform proof).
In general we show that there exist infinitely many irreducible
representations $F$ with $\tr F(\tau)\ne 0$.

This theorem is a generalization of the results of A.Borel,
J.P.~Labesse and J.Schwermer \cite{BLS}. They prove such a result for
an almost absolutely simple, connected, algebaic group $\bG$ and a
Cartan-like involution  $\tau$.}


\subsection{} In the special case of $\bG =GL_n$ we consider the
involution $\tau_c$ with fixed points $SO(n),$ and if  $n = 2m$
also the the symplectic involution $\tau_s$ with fixed points
$Sp(n).$ The following theorem summarizes our results for these special
cases.

\begin{theorem}[theorems (1) and (2) in \ref{sec:10.2}]
There exist cuspidal representations $\pi_\bA $ of GL(n,$\bA$)
with trivial infinitesimal character invariant under
the Cartan involution $\tau_c$. If n= 2m there also exist cuspidal
representations $\pi_\bA $ of GL(n,$\bA$) with trivial
infinitesimal character invariant under $\tau_s$.
\end{theorem}

A suitable generalization can be proved for the case when the
infinitesimal character of $\pi_\bA$ coincides with that of a finite
dimensional representation $F.$

\medskip
Using base change results of J.~Arthur and L.~Clozel \cite{AC}, we
obtain in theorem (3) in section \ref{sec:10.2} cuspidal
representations for GL(n) over number fields which are towers of
cyclic extensions of prime order of $\bQ.$

\subsection{}

 Let $K_f$ be an open compact subgroup of $\bG(\bA_f)$, and
  $A_G$ the
 split component of the center of $\bG(\bA).$
 Then \[S(K_f):= (K_\infty K_f)\backslash {\bG}(\bA)/A_G
\bG(\bQ)\] is a locally symmetric space.

\begin{theorem}[theorem \ref{t:10.3}]
Let $\mathbb G$ be a connected reductive linear algebraic group
defined over $\bQ$ which admits a Cartan like involution.   Then
for $K_f$ small enough
\[
H^*_{cusp}(S(K_f),\bC) \not = 0.
\]
\end{theorem}

\medskip
Previously, nonvanishing results for the cohomology of locally
symmetric spaces were proved by \cite{BLS} for the case of
semisimple groups and $S-$arithmetic groups also using
$L^2-$Lefschetz numbers. Using geometric techniques results of
this type for an anisotropic form of $\bG$ were proved in \cite{RS}
and in the special case of the SO(n,1). (For a more detailed
history of the problem see section \ref{10.3})

\bigskip

\subsection{}{Our main tool is the twisted Arthur trace
formula. We construct a function $f_\bA$ which satisfies the
conditions for the simple trace formula of Kottwitz/Labesse to hold.
The major part of the article is devoted to analyzing the twisted
orbital integrals of this function.
\medskip

In the real case the main result is the following. We first prove
in theorem \ref{t:5.3} a formula for the Lefschetz numbers of the
automorphism $\tau$ on the $(\fk g,K_\infty)-$cohomology of
standard representations { and define a Lefschetz function $f_F$. We use this to find an explicit formula
for the twisted orbital integral $O_\gamma(f_F)$  of an
arbitrary elliptic element $\gamma=\delta\tau$ in \ref{t:6.4.1}.}

\begin{theorem}[theorem \ref{t:6.4.1}]
Let $f_{F}$ be the Lefschetz function corresponding to a
$\tau-$stable finite dimensional representation $F$ and $\gamma=\delta\tau$
be an elliptic element. Then
$$
{ O_\gamma(f_F):=}\int_{\bG(\bR)/\bG(\bR)(\gamma)}f_{F}(g\gamma g^{-1})\
dg=(-1)^{q(\gamma)}e(\tau)\tr F^*(\gamma)
$$
\end{theorem}
The undefined notation is as in section \ref{VI}.

\medskip
At the finite places, the main result is a slight generalization of a
result of Kottwitz for the value of the orbital integral of an
elliptic element $\gamma=\delta\tau$ on a Lefschetz function $f_\C L$.

\begin{theorem}[theorem \ref{t:7.5}]
The orbital integrals of $f_{\C L}$ are
$$
O_\gamma(f_{\C L})=\begin{cases} 1 &\text{ if } \gamma \text{ is elliptic,}\\
                                 0 &\text{ otherwise.}
\end{cases}
$$
\end{theorem}

 \medskip
In the last section we plug the function $f_\bA$ into the
trace formula,  and prove that under the assumption of  theorem
\ref{t:10.1.2} we get a nonzero cuspidal contribution on the
spectral side of the trace formula.

Throughout the article we illustrate the results in the example
of $GL(n)$.}

\medskip
\subsection{}The paper is organized as follows. In sections \ref{I}
and \ref{II} we introduce notation and review basic facts about
twisted conjugacy classes. In section \ref{III} we introduce
orbital integrals, in particular we specify the normalization of
the invariant measures we use. In section \ref{IV} and \ref{V} and
\ref{VI} we deal with finite dimensional representations and
Lefschetz numbers in the real case. The main idea is well known;
for a finite dimensional representation $F$, we construct a
\textit{Lefschetz function} $f_F$ which has the property that for
any representation $\pi,$ $\tr\pi(f_F)$ equals the Lefschetz
number. We rely heavily on the work of Knapp-Vogan, \cite{KnV} and
Labesse \cite{LAB1}. Representations of a disconnected group,
$\wti\bG:=\bG\ltimes\langle\tau\rangle,$ where $\tau$ acts on
$\bG$ by an automorphism of finite order, are described by Mackey
theory. We use the version of the classification of irreducible
$(\fk g,K)$ modules of Vogan, where a standard module is
cohomologically induced, \ie of the form $\C
R^s_{\fk  b}(\chi),$ where $\fk b\subset \fk g$ is a Borel
subalgebra. This makes it convenient to extend modules of
$\bG(\bR)$ to $\wti\bG(\bR)$ in a uniform way by using only the
extension of the
action of $\tau$ to Verma modules. The main result is theorem
\ref{t:5.3} which computes the Lefschetz number of a standard
module. Section \ref{VI} computes orbital integrals
$O_\gamma(f_F).$ These formulas are used in  the
trace formula. Since our methods are global, we have to construct
Lefschetz functions at the finite places as well, and compute
orbital integrals; for this we rely heavily on \cite{KO1}. In section
\ref{VIII} we describe the twisted trace formula, and in \ref{IX}
we give the \textit{simple form} that we actually use. This relies
on not only choosing Lefschetz functions at the infinite places
and two of the finite places, but also choosing the components of
$f_\bA$ to have very small support at a finite set of
places, lemmas \ref{small} and \ref{small enough}.

The main result is in section \ref{sec:10} proposition
\ref{p:10.1} and theorem \ref{t:10.1.2}. Section \ref{sec:10.2} is
devoted to the standard example of $GL(n)$ with the automorphism
transpose inverse, and other applications where $\tau$ does not
necessarily have to be an outer automorphism.

\section{\bf Assumptions and Notation}\label{I}

\subsection{}\label{1.1}
{ Let $\bK$ be an arbitrary number field with Galois
group $\Gamma$. Its adeles are denoted by $\bA$ and the finite
adeles by $\bA_f$. If the adeles refer to a field other than $\bK,$ this will
be indicated by a subscript.

 A localization $\bK_v$ at a place $v$ will be abbreviated $\bk$.  Denote by
$\kbar$ its  algebraic closure, and $\Gamma_\bk $ the Galois group.
{Since $\bK$  is totally real, $\bk$ is $\bR$  at an infinite place,
and a finite extension of $\bQ_p$ for some finite prime $p$ at a
finite place.}

\medskip

Let $\mathbb G$ be a connected reductive linear algebraic group
{defined over $\bK.$ Since $\bK \subset \bk$, $\mathbb
G$ is also defined over $\bk$ }, \ie there is a group
homomorphism
\begin{equation}
\Gamma_\bk\longrightarrow Aut(\mathbb G({\kbar})) \label{1.1.1}
\end{equation}}
which takes regular functions to regular functions in the sense
that if $f$ is regular, then so is $[\sig\cdot
f](x):=\sig^{-1}[f(\sig(x))]$ for any $\sig\in\Gamma_\bk.$

To simplify notation we write $G$ for $\mathbb G(\bk)$ and $\Gsn$
for $\mathbb G(\ovl\bk).$ The connected component of the identity of a
group is denoted by subscript 0.
{The Lie algebra of a subgroup is always denoted by the
corresponding gothic letter.}

{ We assume that the derived group ${\mathbb G}_{der}$ is simply
connected so that we have to deal with fewer technicalities.} This is
certainly true for the main example we have in mind which is
${\mathbb  G}=GL(n)$ whose derived group is $SL(n).$ But in
general we will also consider groups of the form  $\Gsn(\tau)$, the
centralizer of an element $\tau\in \Gsn$. Such a group may not be
simply connected, and in fact may not even be connected. A result of
Steinberg states that if $\bG$ is simply
connected, then $\bG(\tau)$ is connected.

\subsection{}\label{1.2}
{We will fix an automorphism $\tau$ of finite order $d$}
{\it defined} over $\bK$ (\ie
$\tau\sig=\sig\tau$ for all $\sig\in \Gamma$).

Denote by $\wti G$ the group
\begin{equation}
\wti \bG:=\bG\ltimes\{1,\tau\},\qquad
[g_1,a_1]\cdot[g_2,a_2]=[g_1\tau(g_2),a_1a_2],
\label{1.1.2}\end{equation} 
and write $\bG^*:=\bG\tau$ for the
connected component containing $\tau.$ The center of $\mathbb G$ is
denoted $\C Z,$ and the center of $\Gsn$ by $\mathbf Z.$

{ For $\gamma \in G^*$  we denote its centralizer in $G$ by
$G(\gamma )$.}

 \subsection{}\label{1.3} If $\gamma\in
\tilde{G}$ is arbitrary, then it has a Jordan decomposition
$\gamma=\gamma_{ss}\gamma_n$. Then $\gamma_n=e^N$ where $N\in\fk
g(\gamma)$ is nilpotent. In particular $\gamma_n\in G,$ and
$\gamma_s\in G^*.$

{
\begin{definition}\label{d:1.3.1}
{ An element $\gamma$ is called semisimple if $\Ad\gamma$
  is semisimple.} An element $\gamma\in \wti G$ is called almost semisimple
if it stabilizes a pair $(\fk b,\fk h)$ where $\fk b$ is a Borel
subgroup of the Lie algebra $\fk g$ of $\mathbf G$ and $\fk h\subset
\fk g$ is a Cartan subalgebra. For such elements, $\Ad\gamma$ is
semisimple on $[\fk g,\fk g],$ but the action on the center of $\fk g$
need not be semisimple.

A semisimple   $\gamma \in \tilde{G}$ is called
\begin{enumerate}
\item elliptic if $G(\gamma)$ contains a maximal anisotropic torus,
\item compact if the closure of the group $<\gamma>$ generated by
  $\gamma$ is compact 
\item regular if it is in $G$ and its centralizer in $G$ is a
  torus,
\item superregular if it is of the form $s\tau$ with $s\in G$, and \newline
$s\tau(s)\tau^2(s) \dots \tau^{d-1}(s)=(s\tau)^d$ is regular.
\end{enumerate}
\noindent The set of {\it superregular} points is denoted by $\Gsr.$

\noindent{The map
\begin{equation}\label{eq:1.3.1}
N:s\tau
\rightarrow \ s\tau(s)\tau^2(s) \dots \tau^{d-1}(s)
  \end{equation}
 is called the norm map. }

\end{definition}
}
\subsection*{Remarks}
\begin{enumerate}
\item In the $p-$adic case the definition of elliptic is equivalent to the
closure of the group $<\gamma>$ generated by $\gamma$ being compact (see
\cite{KO} section 9.1 or \cite{KOSHE} section I), but not in the real
case.
{
\item
In section 3.2 of \cite{KOSHE} the norm map is defined for an abelian
group $H$ as the canonical quotient map 
\begin{equation}\label{l:r1}
H\longrightarrow H/(1-\tau)H.
\end{equation} 
where $(1-\tau)H:=\{ x\tau(x^{-1})\ :\ x\in H \}.$ 
See lemma \ref{l:2.1} for a comparison between $H^\tau$ and 
the image of the map in \ref{eq:1.3.1} over a closed field.\qed
}
\end{enumerate}


\medskip
The set $\wti G_{sr}$ is open and dense in $G^*$.  

\medskip
{ For $\gamma \in G^*$ its G-conjugacy class is denoted
by $O(\gamma)$. Its $\bf G$-conjugacy class is denoted by ${\bf
O}(\gamma)$.}

\section{\bf Twisted conjugacy classes}\label{II}

\noindent {In this section we discuss the following:

\begin{enumerate}
\item  describe the conjugacy classes of elements $\gamma\in
G^*$ under the adjoint action of $G.$ These are called twisted
conjugacy classes.

\item set up a map from twisted conjugacy classes in $\G^*$ to
usual conjugacy classes in $G(\tau).$ This is called a norm class
correspondence. We also discuss this map in detail for $\mathbb G=
GL(n)$.
\end{enumerate}
}
For the real case many or most of these results are well known from
the work of Bouaziz. We provide proofs that are uniform for both the
real and the p-adic case.

\medskip
\subsection{}\label{2.1} We first work in $\Gsn:=\mathbb G(\kbar).$
{Suppose that $\gamma\in \Gsn^*$ is } semisimple.
According to \STa, there is a pair $(\fk b,\fk h)$, where $\fk b$
is a Borel subgroup and $\fk h\subset \fk b$ is a \CSA, which is stabilized
by $\gamma.$ Since $\tau$ is semisimple ($\tau^d=1$), fix a
$\tau$-stable pair $(\fk b_0,\fk h_0)$. Then there is $g\in
\Gsn$ such that $\fk b_0=g\fk b g^{-1},\ \fk h_0=g\fk hg^{-1}.$
Thus {replacing  $\gamma$ by $g\gamma g^{-1}$ } we may
assume that $\gamma$ stabilizes a $\tau$-stable pair $(\fk b_0,\fk
h_0)$, \ie there is a pair $(\fk b_0,\fk h_0)$ such that
\begin{equation}
\gamma(\fk b_0)=\fk b_0,\ \gamma(\fk h_0)=\fk h_0\quad\text{ and }
\quad \tau(\fk b_0)=\fk b_0,\ \tau(\fk h_0)=\fk h_0.
\label{2.1.1}\end{equation}
 It follows that if we write
$\gamma=h\tau,$ then $h\in \Hb,$ the \CSG\   corresponding to
$\fk h.$ Write $\T$ for the fixed points of $\tau$ in $\Hb.$

{\begin{lemma}\label{l:2.1}
  Let
$$
\Hb^\perp:=\{ h\in \Hb\ :\ h\tau(h)\cdot\ \dots\
\cdot\tau^{d-1}(h)=1 \}.
$$
Then
$$
\Hb^\perp=(1-\tau)\Hb:=\{ h\tau(h)^{-1}\ :\ h\in \Hb\},
$$
and the  map
\begin{equation}
\Psi: \T\times \Hb^\perp \longrightarrow \Hb,\qquad
\Psi(t,h):=th^{-1}\tau(h) \label{2.1.2}
\end{equation}
is onto and has finite kernel.
\end{lemma}
\begin{proof}
  Let $\C X_*$ be the weight lattice. Since $\Hb$ is abelian and connected,
$\Hb=\C X_*\otimes_{\bZ}\kbar^\times,$ and
 $\T=\C X_*^\tau\otimes_{\bZ}\kbar^\times.$ To show that $\Psi$ is onto, observe
that the polynomial relation
\begin{equation}
  \label{eq:II.1.3}
  T^{d-1}+\dots +1=[T-1]\cdot[T^{d-2}+2T^{d-3}+\dots + (d-1)] +d
\end{equation}
holds. Thus any $x\in \C X_*$ can be written as
\begin{equation}
  \label{eq:II.1.4}
x=\frac{x + \tau(x) +\dots +\tau^{d-1}(x)}{d} + \frac{y-\tau(y)}{d},
\end{equation}
where
\[
y:=-(d-1)x -(d-2)\tau(x)-\dots -2\tau^{d-3}(x) -\tau^{d-2}(x)\in \C X_*,\]
Write
\[z:=x + \tau(x) +\dots +\tau^{d-1}(x)\in \C X_*^\tau.
\]
Passing to $\bf H,$ let $h:=x\otimes a,$ with $x\in \C X_*$ and $a\in\ovl{\bf
  k}^\times,$ and let $\al\in\ovl{\bk}^\times$ be such that $\al^d=a.$ Then
\begin{equation}
  \label{eq:II.1.5}
  x\otimes a=(z\otimes\al) (y\otimes\al)\tau(y\otimes\al)^{-1}
\end{equation}
where $z\otimes\al\in \bf T,$ and
$(y\otimes\al)\tau(y\otimes\al)^{-1}\in \bf H^\perp.$
The fact that the map has finite kernel now
follows from computing the differential of $\Psi,$ and observing that
it is nondegenerate.
\end{proof}

\medskip
\begin{proposition}\label{p:2.1} Every semisimple $\gamma\in \bf G^*$
  is conjugate   under $\bf G$ to an element of the form $t\tau$ with
  $t\in\bf  T.$ 
\end{proposition}
\begin{proof}
By the discussion at the beginning of section
\ref{2.1}, there is an element $g\in G$ such that $g\gamma g^{-1}$
stabilizes the pair $(\fk b_0,\fk h_0).$ Thus $\wti\gamma:=g\gamma
g^{-1}$ is of the form $h\tau$ with $h\in \bf H.$ By lemma
\ref{l:2.1}, $h\tau=h^\perp t\tau=x\tau(x)^{-1}t\tau=xt\tau x^{-1}.$
\end{proof}

\medskip
\begin{corollary} {Suppose that $\gamma \in \mathbf G^*$ is semisimple and
 $\gamma^d=1.$ } Then $\gamma$ is conjugate by $\mathbf G$ to $\tau.$
\end{corollary}
\begin{proof}
By the above discussion, $\gamma$ is conjugate to an element of
the form {$t\tau$ with $t$ in a torus $ \T$ fixed under
$\tau$}. On the other hand, the condition $\gamma^d=1$ implies
that $t\in \Hb^\perp.$ Thus $t=h^{-1}\tau(h)$ so
$$
t\tau=h^{-1}\tau(h)\tau=h^{-1}\tau h.
$$
\end{proof}

\subsection{}\label{2.2}
We now consider the case of $\bf k$ which is not necessarily
closed. Recall $\Hb =\T \Hb^\perp \subset \G$, a $\tau$-invariant Cartan
subgroup.

\medskip
\begin{proposition}[1]\

\begin{enumerate}
\item $\T$ contains regular elements { as well as
    superregular elements}.
\item  If $\T$ is defined over $\bk,$ then { $T:=\bT(k)$ contains
regular elements as well as superregular elements.}
\end{enumerate}
\end{proposition}
\begin{proof}  {
We first show the assertions for $\bb T.$ 
Because we are working over $\kbar,$ it is enough to
   consider the case of the
Lie algebra; then it is the same as over $\bb C$ and it is due to
F. Gantmacher ( See theorem 5 and 28 in \cite{Gant}). Thus
the set of regular elements as well as the set of superregular
elements is dense in $\T.$ For $T,$ the assertions follow
from Rosenlicht's density theorem, which  says that for a perfect
field $k$, $\bG(\bk)$ is dense in $\bG(\kbar)$ for any $\bG$ defined
over $\bk$, \cite{Ros}.}
\end{proof}

\noindent {\bf Remark.}\ 
In  case $\bk=\bb R,$ a similar result is proved by
A.  Bouaziz (1.3.1 in \cite{Bou}) for a more general type of group in
essentially in the same way.

\begin{proposition}[2] There exist finitely many superregular elements \newline
  $\gamma_1,\dots, \gamma_k\in G^*$ such that every superregular element is
  conjugate under $G$ to an element in ${T}_i\gamma_i,$ where
  $T_i:=G(\gamma_i).$
\end{proposition}

\begin{proof}
By proposition (1) of \ref{2.2}, there are superregular elements defined
over $\bf k;$ in fact the set of such elements is open and dense in
$G^*.$ Fix a superregular $\gamma\in G^*,$ and let $(\fk b,\fk h)$ be
a pair stabilized by $\gamma.$ Let $\bb T$ be the centralizer of
$\gamma.$  Then $\bf T$ is
stabilized by any $\sig\in \Gamma$ because $\gamma$ is stabilized by
such a $\sig.$ Since $\bf H$ is the centralizer of $\bf T,$ 
the  same holds for $\bf H$.
Thus there is a Cartan subgroup $H\subset G$ defined over $\bf k$ 
which is normalized by $\gamma.$  
{
Two superregular elements $\gamma_1$ and $\gamma_2$
which stabilize the pair $(\fk b,\fk h)$ differ by an automorphism of
the pair; there are only finitely many such automorphisms. Suppose
$\gamma_1$ and $\gamma_2$ induce the same automorphism.  Then
 ${H}':={T}\{h^{-1}\gamma(h)\}_{h\in {H}}$ has finite
index in $H.$ Let $a_1,\dots ,a_l$ be representatives of
${H/H'}.$ It follows that $\gamma_2$ is conjugate to $ta_j\gamma_1$ for some
$1\le j\le l.$ The claim now follows from the fact that the number of
$G-$conjugacy classes of \CSGs\ and pairs $(\fk b,\fk h)$ with the
same $\fk h$ is finite.
}

If a superregular $\gamma'$ normalizes the Cartan subgroup $H,$ then
$\Ad\gamma'$ is an automorphism of $H$ which stabilizes the
root system, and there are only finitely many such automorphisms.

Then $\gamma_2=h\gamma_1$ with $h\in {H}$. Let ${T}$ be
the fixed points of $\gamma_1$ (and $\gamma_2$ as well). 
\end{proof}

\subsection{}\label{2.3}
Recall (\ref{eq:1.3.1}) the \textit{norm map}, $N:\bf
G^*\longrightarrow G,$ given by $N(\gamma):=\gamma^d$. It induces a
map from twisted
$\Gsn$-orbits in $\Gsn^*$ to usual orbits in $\Gsn.$ The
discussion above shows that for $\gamma$ semisimple, the  orbit of
$N(\gamma)$ intersects the fixed points $\Gsn(\tau).$ Thus $N$ induces a
a \textit{norm class correspondence} $\C N$, which takes twisted
semisimple orbits in $\bf G^*$ to unions of semisimple orbits in
${\bf G}(\tau):$
\begin{equation}
\C N: {\bf O}(\gamma)\mapsto  {\bf O}(\gamma^d)\cap \Gsn(\tau).
\label{2.3.1}
\end{equation}
\begin{proposition} If $\gamma$ is semisimple, then
${\bf O}(\gamma^d)\cap \Gsn(\tau)$ is a finite union of orbits.
\end{proposition}
\begin{proof}  From earlier, $\gamma$ is
conjugate to $t\tau$ with $t\in {\bf T}$ showing that the
intersection is nonempty.  The orbit ${\bf O}(\gamma^d)$ coincides
with ${\bf O}(t^d).$ Thus it suffices to show  that ${\bf O}(t^d)$
intersects $\Gsn(\tau)$ in finitely many orbits. But this is clear
since ${\bf O}(t^d)\cap \bf H$ is finite.
\end{proof}}

The map $N$ and correspondence $\C N$ make sense for $G^*$ and $G.$
It is still true that $O(\gamma^d)\cap G(\tau)$ consists of finitely
many orbits, but the image might be empty.

\bigskip

\subsection{}\label{2.4} {We illustrate this for
$\mathbf G=GL(n,\ovl{\bk}).$} The automorphism $\tau$ will be of order two
related to Cartan involutions of real groups. We write it as
\begin{equation}\label{2.4.1}
\tau(g):=w_0(\ ^tx^{-1})w_0^{-1},\qquad
w_0=\begin{bmatrix} 0&\dots &\dots&\dots &\dots  & 1\\
                    0&\dots &\dots &\dots &1     & 0\\
                     \vdots&\vdots &\vdots&\vdots &\vdots&\vdots\\
                     0&1    &\dots&\dots  &\dots & 0 \\
                     1&\dots &\dots &\dots &\dots& 0\\
     \end{bmatrix}
\end{equation}
{If $n$ is odd $\tau$ is up to conjugation  the unique
outer automorphism of order two}. If $n$ is even, there is
another conjugacy class of automorphisms of order two with
representative {$\Ad(t_o)\circ\tau$} where
$$
t_0=\diag[i,\dots,i,-i,\dots ,-i].
$$
 The element $t_0\tau$ has
order 4, and is conjugate to the one in more familiar form where
$w_0$ in (\ref{2.4.1}) is replaced by
\begin{equation}\label{2.4.2}
w_0=\begin{bmatrix} 0&\dots &\dots&\dots &\dots  & 1\\
                    0&\dots &\dots &\dots &1     & 0\\
                     \vdots&\vdots &\vdots&\vdots &\vdots&\vdots\\
                     0&-1   &\dots&\dots  &\dots & 0 \\
                    -1&\dots &\dots &\dots &\dots& 0
     \end{bmatrix}.
\end{equation}
These elements preserve the upper triangular group which gives a
$\tau$ stable pair $(\fk b_0,\fk h_0)$ defined over $\bk.$ The
centralizers are $\Gsn(\tau)=O(n)$ for (\ref{2.4.1}) and $\Gsn(\tau)=Sp(n)$
for (\ref{2.4.2}). In  case (\ref{2.4.1}) with $n$ even, the
automorphism $\tau$ does not preserve the root vectors. In other
words it does not preserve a splitting.

More generally, outer automorphisms of finite order of
$GL(n,\kbar)$ are (conjugate to elements) of the form $t\tau$ with
$t\in \bf H$ fixed by $\tau$. The centralizer of such an element is a
product of $GL(m)'$s and possibly an orthogonal or a symplectic
group. This follows from the fact that the centralizer of
$t$ is a product of $GL(m)'$s.

\medskip
Conjugacy classes of semisimple elements in $GL(n,\bk)$ are
determined by the characteristic polynomial
\begin{equation*}
p_A(\la)=\det(\la I-A).
\end{equation*}
The minimal polynomial $m_A(\la)$ of $A$ is
the polynomial of minimal degree with leading coefficient 1 satisfying
$m_A(A)=0.$
The matrix $A$ is  { regular} if $m_A=p_A,$ where $m_A$ is
the minimal polynomial of $A.$   Two regular
semisimple matrices $A$ and $B$ are similar if and only if
$m_A=m_B.$

\bigskip
Suppose $\gamma=x\tau\in{\bf\wti G}_{sr}$ {\it is superregular.} From
\ref{2.1} we know that we can assume $x\in \bf T,$ so $\gamma^2$ can
be conjugated into $\bf G(\tau).$
\begin{proposition}[1]
Suppose $\gamma\in{\bf \wti G}_{sr}.$ Then ${\bf O}(\gamma^2)\cap {\bf
G}(\tau)$ is a single orbit.
\end{proposition}
\begin{proof}
Two regular semisimple elements $A,\ B\in {\bf G(\tau)} $ are
conjugate if and only if
$p_A=p_B.$ This follows by using an explicit realization of a \CSG\ 
of $\bf G(\tau).$ (Note however that this fact is {\bf not}
true for $SO(\kbar)$.) Thus all the elements in
$\bf O(\gamma^2)\cap \bf G(\tau)$ have the same minimal polynomial
and the claim follows.
\end{proof}

\medskip

 We will now  show that
$\mathbf O(\gamma^2)$ has points defined over $\bk.$ For this we will use
the cross section introduced by Steinberg. Let $\Gsn$ be any connected
linear algebraic semisimple group, and $(\mathbf B,\mathbf H)$ be as in section
\ref{2.1}. Let $\Delta^+$ be the system of positive roots attached
to $(\mathbf B,\mathbf H)$ and $\Pi$ be the simple roots. Let $\sig_i$ be
representatives for the simple root reflections, and $X_i$ be root
vectors. Let  $Z_i(t)=exp(tX_i)$ be the corresponding 1-parameter
subgroup. In section 1.4 of  [Steinberg2]) it is proved that the set
\begin{equation}
\C M:=Z_1(t_1)\sig_1\cdot\ \dots\ \cdot Z_n(t_n)\sig_n \label{2.4.3}
\end{equation}
is a cross section for the set of regular elements of $G.$

\noindent If $G$ is defined over $\bk$ and is quasiplit, then $\C
M$ can be defined over $\bk.$

\medskip
 \noindent\textbf{Remarks.}
\begin{enumerate}
\item There is a restriction on $n$ for type $A$ for the fact that $\C
  M$ is defined over $\mathbf k.$ This restriction applies to the
  unitary groups, not the general linear group.
\item  For $SL(n,\bk),$ this cross section is the well known canonical form
\begin{equation}
\C M=\{\  \bmatrix c_{n-1}&-c_{n-2}     &\dots &(-1)^{n-1}c_1   &(-1)^n       \\
                 1&0     &\dots &0     &0   \\
            \vdots&\vdots&\dots &\vdots&\vdots  \\
                 0&0     &\dots &0     &0 \\
                 0&0     &\dots &1     &0    \endbmatrix\ \}.
\end{equation}
The characteristic polynomial is $p(\la)=\la^n-c_{n-1}\la^{n-1}+\dots
+(-1)^n.$
\end{enumerate}

\medskip
Since $\gamma^2$ can be conjugated into $\bf G(\tau)$ the
characteristic polynomial of $\gamma^2=x\tau(x)$ satisfies
\begin{equation}
\la^np_A(\la^{-1})=p_A(\la). \label{eq:2.4.5}
\end{equation}

\begin{proposition}[2] Suppose that either $n$ is odd and $G(\tau)=O(n)$ or
  $n$ is even and $G(\tau)=Sp(n).$ For any polynomial $p(\la)$ satisfying
(\ref{eq:2.4.5}) with coefficients in $\bk,$ there is a regular
element $A$ in the split form of $G(\tau)$ with characteristic
polynomial $p.$
\end{proposition}

\begin{proof} Write $n=2m+\ep.$
We can choose a labelling of the simple roots so that the one
parameter subgroups $Y_i(t)$  of the simple root vectors in $G(\tau)$
are
\begin{equation*}
\begin{aligned}
&Y_1(t_1):=Z_1(t_1)Z_2(t_1),\dots
,Y_{m-1}(t_{m-1}):=Z_{m-1}(t_{m-1})Z_m(t_{m-1}),\\
&Y_\ep(t_\ep):=Z_\ep(t_\ep)
\end{aligned}
\end{equation*}
and the $\sig_i$ have similar expressions. The proof follows from a direct
calculation. Any polynomial $p(t)=t^n-c_1t^{m-1}+c_2t^{m-2} +\dots $ satisfying
\ref{eq:2.4.5} is the characteristic polynomial of the matrix obtained
by setting $t_1=c_1,\dots ,t_{m-1}=c_{m-1},\ t_\ep=c_m.$
\end{proof}

\medskip
Now assume $n=2m$ and that $G(\tau)=O(n)$. Consider the split form
of $O(n)$. In this case we can choose a labelling of the simple
roots so that
$$
Y_1(t_1):=Z_1(t_1)Z_2(t_1),\dots ,
Y_{m-1}(t_{m-1}):=Z_{m-1}(t_{m-1})Z_{m}(t_{m-1})
$$
are the 1-parameter subgroups corresponding to $m-1$ of the simple
roots of type $D.$
The last simple root vector $Y_m(t_m)$  cannot be written in terms of
simple root vectors of $GL(2n,\bf k)$. Let $t_m$
be the parameter for the last simple root. The characteristic
polynomial of an element in this cross section is
\begin{align}
  \label{eq:2.4.6}
\ &x^{2m} + a_mx^{2m-1} + (a_{m-1}-1)x^{2m-2} +
(a_{m-2}-a_m)x^{2m-3}
  +\dots \notag \\
\ &-(a_{4} +a_{2}a_{1})x^{m+1} -(2a_{3} + a_{1}^2 + a_{2}^2)x^m
+\dots
\end{align}
All but the last two equations are linear so we can solve for
$a_3,\dots , a_{m}$ in terms of the coefficients of the polynomial
for $A.$ For the last two we need that
\begin{equation}
  \label{eq:2.4.7}
  c_m-2a_3 \pm(2c_{m-1}-2a_4)
\end{equation}
be squares.
\begin{proposition}[3]
  If $n=2m,$ the orbit $O(\gamma^2)$ intersects a unique quasi split
  form of $O(n).$
\end{proposition}
\begin{proof}  Let $\zeta_1,\ \zeta_2\in \bk.$
Consider the orthogonal group which preserves the form
\begin{equation}
Q(x):=x_1x_{2m} + \dots + x_{m-1}x_{m+2}+\zeta_1x_m^2 +
\zeta_2x_{m+1}^2. \label{2.4.8}\end{equation} These groups are
quasisplit. Precisely, the most split \CSG\  is
\begin{equation}
H:=\{\ diag(a_1,\dots ,a_{m-1},\begin{bmatrix} a&b\\-b\zeta &
a\end{bmatrix},a_{m-1}^{-1},\dots , a_1^{-1}),\} \label{2.4.9}
\end{equation}
where $\zeta=\zeta_1/\zeta_2,\ a^2+\zeta b^2=1.$

A similar calculation in the proof of the previous theorem shows
that there is always a choice of $\zeta_1,\ \zeta_2$ so that the
equations have a solution.
\end{proof}

\medskip

If a regular semisimple $t\in GL(n,\bk)$ is such that its orbit is
$\tau-$stable, then $t$ is conjugate to a $\tau-$stable
element by a $g\in GL(n,\bk)$ not just $GL(n,\kbar).$ This is
again because the characteristic polynomial determines the
conjugacy class over $GL(n,\bk).$

\begin{corollary} Every semisimple $\gamma\in G^*$ is
conjugate by an element of $GL(n,\bk)$ to one of the form $x\tau,$
wth $x$ in a $\tau$-stable \CSA\  $\ H$.
\end{corollary}
\begin{proof} The centralizer $C(\gamma,\kbar)$ contains a
superregular point whose $G$-orbit is $\tau$-stable, because by
section \ref{2.1} $\gamma$ is conjugate to an element $h\tau$ with
$h\in H_0.$ Because the variety of such elements is invariant
under $\Gamma,$ it has regular rational points in $C(\gamma,\bk).$
Let $t$ be such a point. Then conjugate $t$ via $GL(n,\bk)$ into a
$\tau$-fixed point. So we may as well assume that $C(\gamma,\bk)$
has a $\tau$-fixed regular semisimple point $t.$ Let $H$ be the
(necessarily $\tau$-stable) \CSG\ corresponding to $t.$ Then since
$\gamma=x\tau$ centralizes it, $x\in H.$
\end{proof}
In particular, all the conclusions about the norm map and regular
elements, extend to the case of semisimple elements.

\medskip
\subsection{Orbits in the Real Case}\label{2.5}

{ Let $G:=\Gsn(\bb R)$ be the real points of $\mathbb G.$
We assume that $\tau$ is defined over $\bR$, so it induces an
automorphism of $G$.  We denote by subscript $0$ a real algebra; an
absence of a subscript indicates a complex algebra or vector space. So
let  $\fk g_0$ be the Lie algebra of $G.$ It is well
known ([Helgason]) that we may fix a maximal compact subgroup K,
and a Cartan decomposition $\fk g_0=\fk k_0 +\fk s_0$ with Cartan
involution $\theta$ so that $\theta $ commutes with $\tau$.}  Then
the Cartan decomposition $\fk g_0=\fk k_0 +\fk s_0$  is invariant under
$\tau.$ Let $\widetilde K:=K\ltimes \{\tau\}.$

\medskip
If $\gamma \in G^*$ is semisimple, it has a decomposition into its
compact and hyperbolic parts $\gamma=\gamma_{e}\gamma_{h}.$
Here $\gamma_{h}=exp Y$ where $Y$ is hyperbolic.

Suppose $\gamma$ is compact. There is a Cartan involution $\theta'$ which
commutes with $\gamma$ (\Helg). Then let $g\in G$ be such that
$g^{-1}\theta g=\theta'.$ Then $g\gamma g^{-1}$ is fixed by
$\theta.$ So if we write $\gamma=x\tau,$ then $x\in K.$

So in general, if $\gamma=\gamma_e\gamma_h$ is arbitrary, we may
assume (by possibly conjugating $\gamma$ by $G$) that
$\theta(\gamma_e)=\gamma_e$ and $\theta(Y)=-Y.$ We
will do so without further mention.

\medskip
 The classification of compact elliptic elements reduces to the
corresponding problem for the compact group. But since this group
may be disconnected, we need some modification of our previous
results.

We assume in this section that $K$ is arbitrary compact with
identity component $K_0,$ and Lie algebra $\fk k _0$.
Denote by $K_c$ the \textit{connected}
group with Lie algebra $\fk k,$ { the
complexification  of $\fk k_0$ .} We say that a pair $(\fk b,\fk
h)$ in $\fk k$ is rational if $\fk h$ is rational (or
equivalently $\fk b\cap\theta\fk b=\fk h$).

\begin{proposition} {Suppose that   $\gamma \in K\tau.$}
There is a  rational pair $(\fk b,\fk h)$ stable  under $\gamma.$
\end{proposition}
\begin{proof}
Let $(\fk b,\fk h)$ be a fixed rational pair.
There is $k\in K_0$ such that $k\gamma$ stabilizes $(\fk b,\fk
h).$ Let $B\subset K_c$ be the Borel subgroup with Lie algebra
$\fk b$. The map
\begin{equation}\label{2.5.1}
\psi:K_0\times B\longrightarrow
K_c,\qquad\psi(x,b):=xbk\gamma(x^{-1})k^{-1}
\end{equation}
is onto. The pair $(x\fk b,x\fk h)$ satisfies the required
properties.
\end{proof}
{Recall that for an arbitrary compact group, a \CSG\ $H\subset K$ is
defined to be the normalizer of a rational pair  $(\fk b,\fk h)$.
Similarly $\wti H$ is the normalizer of $(\fk b,\fk h)$
in $\wti K.$ }
\begin{corollary}\label{c:2.5} { Suppose that   $\gamma \in K\tau.$}
  Then $\gamma$ is conjugate via $K_0$ to an element which leaves a
  $\tau$-stable rational pair $(\fk b_\tau,\fk h_\tau)$ invariant.
  Thus $\gamma$  is conjugate to an element of the form
  $\gamma=h\tau$ with $h\in H.$ Any element of $\wti K$ is conjugate
  via $K_0$ to an   element in $\wti H,$ and $\wti H$ meets every
  connected component of   $\wti K.$
\end{corollary}
\begin{proof}
  Let $(\fk b,\fk h)$ be a rational pair stable under $\gamma,$ and
  $(\fk b_\tau,\fk h_\tau)$ a rational pair stable under $\tau.$  There
  is $k\in K_0$ such that $\Ad k (\fk b, \fk h)=(\fk b_\tau,\fk h_\tau).$
  Then $k\gamma k^{-1}=h\tau.$ It is clear that $h\in H,$ and
  therefore $k\gamma k^{-1}\in\wti H.$ The claims of the corollary follow. 
\end{proof}

\smallskip
\subsection{}\label{2.6} Results about twisted orbits are often
expressed in terms of group cohomology. Let $\C G$ be a group
acting on another group $\C A.$ A cocycle is a map
\begin{equation}
  \label{eq:2.6.1}
\psi : \C G \longrightarrow \C A,\ \text{ satisfying }\ \psi(st)
=\psi(s)\cdot s(\psi(t))
\end{equation}
Two cocycles $\psi,\ \psi'$ are called equivalent if there is
$g\in \C A$ such that
\begin{equation}
  \label{eq:2.6.2}
  \psi'(s)=g\psi(s)s(g)^{-1}.
\end{equation}
The quotient space of cocycles under this relation is the
cohomology group $\Ho^1(\C G,\C A).$

\medskip
There are two instances where this construction arises. In the
first case, let $\C G=\langle \tau\rangle,$ the group generated by
an automorphism $\tau$ of $\C A.$ Then a cocycle $\psi$ is determined
by its value $\psi(\gamma):=a$ in $\C A.$ If $\tau$ is of finite
order $d,$ then
$$
a\cdot\tau(a)\cdot\ \dots\ \cdot\tau^{d-1}(a)=1.
$$
\begin{proposition}[1]
The map $a\mapsto a\tau$  is a bijection
$$
\Ho^1(\C G, \C A)\longleftrightarrow \{x\in \C A^* \}/\C A.
$$
In words, $\Ho^1$ parametrizes twisted conjugacy classes of
elements in $\C A^*.$
\end{proposition}
\begin{proof}
  We omit the details which are standard.
\end{proof}

\medskip\noindent
In the second instance, let $\C G=\Gamma,$ the Galois group of
$\ovl{\bk}/\bk.$ In this case, $\Ho^1$ is denoted $\Ho^1(\bk,\mathbf G)$
and is called Galois cohomology. Recall that if $\mathbb G$ is reductive
connected simply connected or $GL(n)$, and $\bk$ is a p-adic field,
these groups are trivial.

\begin{proposition}[2] \label{p:2.6} Let
${\mathbf O}(\gamma)$ be the ${\mathbf G}$ orbit of $\gamma\in
G^*.$   Then
$$
[{\mathbf O}(\gamma)\cap G^*]/G \longleftrightarrow
\ker[\Ho^1(\bk,\mathbf G(\gamma))\longrightarrow \Ho^1(\bk,\mathbf G)].
$$
\end{proposition}
\begin{proof} This is well known. An element $x\gamma x^{-1}$ is
$\Gamma$-stable  $s(x\gamma x^{-1})=x\gamma x^{-1}$ for all
$s\in\Gamma$ which is equivalent to $x^{-1}s(x)\in {\mathbf G}(\gamma).$ 
It is clear 
that $\psi(s):=x^{-1}s(x)$ is a cocycle with trivial image in
$\Ho^1(\Gamma, G)$. This cocyle depends only on the $G$ coset of
$x.$ Conversely any cocycle in the kernel must be of the form
$\psi(x)=x^{-1}s(x)$ for some $x\in {\mathbf G.}$
\end{proof}

\bigskip

\section{\bf Orbital Integrals}\label{III}
{ In this section we discuss twisted orbital integrals.
These results will be used in section VI.}

\medskip
\subsection{}\label{3.1}
{Recall from section \ref{II} that if $\gamma\in G^*$ is
  superregular, it can conjugated (by $\mathbf G$) into an element of
  the form $\gamma   = t\tau$, where $t\in \bf T $  is a semisimple element of a
  $\tau$-invariant rational maximal torus of $\mathbf G(\tau)$.}
 As before, let $\mathbf H$ be the centralizer of $\mathbf T,$ a maximal
torus of $\mathbf G.$

\begin{proposition} For any compact set
$\omega\subset G^*,$ there exists a compact set $\Omega\subset
G/G(\gamma)$ satisfying the condition that if $g\gamma g^{-1}\in
\omega,$ then $g\in \Omega G(\gamma).$
\end{proposition}
\begin{proof} 
There is a field $\bk\subset\bk'$ such that $\mathbb H(\bk')$ is
split.  Since  
\[\mathbb G(\bk)/\mathbb G(\bk)(\gamma)\subset \mathbb
G(\bk')/\mathbb G(\bk')(\gamma)\]
is a closed embedding it is
enough to show the claim for the case when $H$ is split. Write
$G=KB$ for a maximal compact $K$ and $B=NH$ a Borel subgroup so
that $G=BK.$ Then  decompose $g=knh.$ The claim follows by
applying the following lemma and the observation that
\begin{equation}
\psi: T\times H^\perp\longrightarrow \gamma H,\qquad (t,h)\mapsto
\gamma h^{-1}\gamma(h) \label{3.1.1}
\end{equation}
has finite fiber and its image has finite index in $H.$ More
details can be found in \cite{Ar} or in the untwisted case in
\cite{HC3} Part I. 
\end{proof}
Assume $P=MN\subset G$ is a rational parabolic subgroup and
  $\gamma$ is rational semisimple such that
  $G(\gamma)\subset M.$ Let $O_M(\gamma)$ be the orbit of $\gamma$ under
$M.$

\begin{lemma}  The map
\begin{align*}
\Psi:& M/G(\gamma)\times N \longrightarrow O_M(\gamma)N,\\
\Psi(m,n):&=m\gamma m^{-1} [Ad(m\gamma m^{-1})^{-1}(n)n^{-1}]
\end{align*}
is an isomorphism.
\end{lemma}
\begin{proof} The proof is identical to the similar result proved by
  Harish-Chandra. The statements are straightforward consequences
of the fact that $d\Psi$ is an isomorphism. 
{ We omit the details which for the untwisted case
can be found for example in \cite{Warner} section 8.1.3, particularly
lemma 8.1.3.6 and corollary 8.1.3.7. 
}
\end{proof}

\medskip
The proposition shows that for $\gamma$ semisimple the
\textit{orbital integrals}
\begin{equation}
  \label{eq:3.1.2}
  \begin{aligned}
&  O_\gamma(f):=\int_{G(\gamma)_0\backslash G} f(g^{-1}\gamma g)\ dg \\
&  I_\gamma(f):=\int_{G(\gamma)\backslash G} f(g^{-1}\gamma g)\  dg
  \end{aligned}
\end{equation}
are  well defined. Following \cite{KO1}, for a connected reductive
group $H$ we use the Euler-Poincare measure. If the group is
disconnected, we use the unique invariant measure which restricts to
the Euler-Poincare measure on $H_0.$

\subsection{}
Harish-Chandra (\cite{HC3} Part II) considered the integrals  
\begin{equation}
  \label{eq:3.1.3}
F^G_f(\gamma):=\nabla(\gamma)^{1/2}\int_{G(\gamma)_0\backslash G}
f(g^{-1}\gamma g)\ dg
\end{equation}
where $\nabla(\gamma):=|\det(I-\Ad\gamma)|_{\fk g/\fk
g(\gamma)}|.$ 
We suppress the superscript $G$ when it is clear
what group is involved.

In the real case we use the following variant of \ref{eq:3.1.3}. 
{
Let $\gamma\in G^*$ be semisimple, and $(\fk b,\fk h)$ 
  a $\gamma-$stable pair. Let $\fk t:=\fk g(\gamma)\cap \fk h.$ Then
  $\fk t$ is a $\Gamma-$stable Cartan subalgebra, and by theorem 1.1A
  in \cite{KOSHE} its centralizer is $\fk h.$  
Define
\begin{equation}
  \label{eq:3.2.1}
'F^G_f(\gamma):=\ 'D(\gamma)O_\gamma(f),\qquad
F^G_f(\gamma):=D(\gamma)O_\gamma(f)
\end{equation}
with
\begin{equation}
  \label{eq:3.2.2}
'D(\gamma):=\prod_{\al\in
\Delta_\gamma^+}[1-e^{\beta}(\gamma)],\qquad
D(\gamma):=\prod_{\al\in
\Delta_\gamma^+}[e^{\beta/2}-e^{-\beta/2}](\gamma).
\end{equation}
The quantities $'D$ and $D$ do not depend on the choice $(\fk b,\fk h),$ and 
the formula for $D$ only makes sense for a cover for which
$e^\rho$ exists, see section \ref{IV}.}
The first asertion follows from the fact that $\gamma$ is conjugate by
an element in $\bf G$ to one of the form $t\tau$ as in proposition
\ref{p:2.1}. 

\medskip

Let $\gamma_1,\dots ,\gamma_k$ be a set of representatives of
conjugacy classes satisfying the conclusion of proposition (2) of section
\ref{2.2}.

\begin{proposition}  Suppose $f\in C_c^\infty(G^*)$ and $\Theta$ is a
locally $L^1$  invariant function analytic on the regular set.
Then
\begin{equation}
  \label{eq:3.3.1}
\int_{G^*} f(g^*)\Theta(g^*)\ dg^*=\sum_i\int_{T_i}
D'(t_i\gamma_i) \Theta (t_i\gamma_i)  F_f(t_i\gamma_i)\ dt_i
\end{equation}
where $D':= \nabla/D.$
\end{proposition}
\begin{proof}
The proof is the same as for the untwisted case. It follows from the
fact that the differentials of the maps
\begin{equation}
  \label{eq:3.3.2}
\Psi_i:  G/T_i\times T_{i}\longrightarrow G^*,\qquad
\Psi_i(g,t):=gt\gamma_i g^{-1}
\end{equation}
are $D'(t_i\gamma_i)$ (so are isomorphisms when restricted to the
regular set), and proposition (2) of section \ref{2.2}. 
{ For the formula in the untwisted case see for
example (IIA) in section 8.1.2 in \cite{Warner}.}
\end{proof}

\subsection{}\label{3.4} Assume $\mathbf k$
is real and that $\gamma\in G^*$ is semisimple. Let $\fk t$ be
a fundamental \CSA\  in $\fk g(\gamma)$ and write $T$ for the
corresponding group. Let $(\fk b,\fk h)$ be a pair which is stable
under $\gamma$ such that $\fk t\subset\fk h.$ Then in fact $\fk h$
is the centralizer of $\fk t$ in $\fk g.$ The roots $\Delta^+$ of
$\fk b$ are stable under $\gamma.$ Decompose $\Delta^+
=\Delta_\gamma^+\cup \Delta^\perp,$ and so $D=D_\gamma\cdot D^\perp$ where $D$
is defined in (\ref{eq:3.2.2}). Let
\begin{equation}
\varpi_\gamma:=\prod_{\beta\in\Delta_\gamma^+} \beta, \label{3.4.1}
\end{equation}
and write $\partial(\omega_\gamma)$ for the corresponding
differential operator. Let $h\gamma$ be superregular, with $h\in
T.$ Then $F_f(\gamma;\partial(\varpi_\gamma))$ is defined to be the
derivative of $F_f(h\gamma)$ in $h,$ and then setting $h=1.$

\begin{theorem} There is a nonzero constant $c(\gamma)$ such
that
\begin{equation*}
F^G_f(\gamma ;\partial(\varpi_\gamma))=c(\gamma)F^G_f(\gamma).
\end{equation*}
\end{theorem}
\begin{proof} Let $t\in T$ be such that $t\gamma$ is regular.
First observe that
\begin{equation}
\partial(\varpi)\circ D^\perp|_{t=1}=D^\perp(\gamma)\partial(\varpi)|_{t=1},
\label{3.4.2}
\end{equation}
because the left hand side is, on the one hand skew invariant
under $W_\gamma,$ on the other hand a linear combination of
constant coefficient differential operators of degree less than or
equal to $\deg\partial(\omega_\gamma).$ Then only the leading term
survives, which is the right hand side of \ref{3.4.2}.

\medskip
On the other hand,
\begin{equation}
F_f(t\gamma)= \int_{G/G(\gamma)}
D^\perp(t\gamma)F_{\Ad(g^{-1})f}^{G(\gamma)}(t\gamma) dg.
\label{3.4.3}\end{equation} The result now follows from
Harish-Chandra's formula \cite{HC2}
\begin{equation}
F_f^{G(\gamma)}(1;\partial(\varpi))=c(\gamma)f(1). \label{3.4.4}
\end{equation}
\end{proof}

\bigskip

{\section{\bf Finite dimensional representations}\label{IV}

 Let $\widetilde K$ be an arbitrary compact group with identity
component $K_0,$ and Lie algebra $\fk k_0$ with complexification $\fk
k.$  Let $(\fk b,\fk h)$ be a pair of a Borel subgroup
$\fk b\subset \fk k$ and a (rational) \CSA\  $\fk h\subset\fk
b.$ Denote by $\Delta^+$ the roots of $\fk h$ in $\fk b,$ and by
$\delta$ all the roots of $\fk h$ in $\fk k.$  Let $\wti
H$ be the Cartan subgroup corresponding to this pair and recall
its properties from section \ref{2.5}. The irreducible finite
dimensional representations of $\widetilde K$ are parametrized by
\textit{highest weights} $\mu$ which are irreducible
representations of $\widetilde H$ with differential dominant for
$\Delta^+.$ { For details about the Cartan Weyl theory of disconnected
  groups see \cite{KnV}.} In this section we obtain a formula for the
character $\tr\pi_\mu(\gamma)$ for
  $\pi_\mu$ the irreducible representation with highest weight
  $\mu$ and $\gamma\in \widetilde H.$ {We also evaluate
    the character on $\gamma $ of order 2.}

   The results of this section
  are known, but since our proofs are different we include them
  here.}

\medskip
\subsection{}\label{4.1}
Let $H_0$ be the connected component of $\widetilde H$ and write
$H$ for the centralizer of $\fk h$ in $\wti H.$   Then
\begin{equation}
  \label{eq:4.1.1}
  H_0\subset H\subset \wti H.
\end{equation}
The Weyl group is defined as $W:=N(\wti H)/H.$ If $w\in W,$ then
$l(w):=\dim \fk b/(w\fk b\cap \fk b)$.

\medskip
We first extend the roots to the group generated by $H_0$ and
$\gamma.$ Let $\Delta^+_\gamma$ be the orbits of the action of
$<\gamma>$ (the group generated by $\gamma$) on $\Delta^+.$ Fix
 root vectors $\{E_\beta\}_{\beta\in\Delta}$.
Let $d(\beta)$ be the size of the orbit of $\beta\in\Delta.$ The vector
\begin{equation}
X_\beta=E_\beta\cdot E_{\gamma\beta}\cdot\ \dots\ \cdot
E_{\gamma^{d(\beta)-1}\beta}\in  S^{d(\beta)}(\fk n)\ \beta\in\Delta^+ \label{4.1.2}
\end{equation}
is independent of the choice of $\beta$ in its $\gamma-$orbit. Define
$e^\beta$ via
\begin{equation}
\Ad(\gamma)X_\beta=e^\beta(\gamma)X_\beta. \label{4.1.3}
\end{equation}
This is independent of the choice of root vectors $E_\beta$ as
well. Similarly let
\begin{equation}
  \label{eq:4.1.2a}
  Y_\beta:=E_\beta\wedge E_{\gamma\beta}\wedge\ \dots\ \wedge
E_{\gamma^{d(\beta)-1}\beta}\in\sideset{}{^{d(\beta)}}{\bigwedge}\fk n.
\end{equation}
Then
\begin{equation}
  \label{eq:4.1.2b}
 \Ad(\gamma)Y_\beta=(-1)^{d(\beta)-1}e^\beta(\gamma)Y_\beta
\end{equation}
Recall that an element $x\in \wti H$ is regular if $[\det \Ad (x)-
I] |_{\fk k /\fk h}\ne 0.$ In particular $e^{\beta}(x)\ne 1$ for
any $\beta\in \Delta^+.$

If $w\in W,$ then $\wti H$ stabilizes $\fk b\cap w\fk b$ and therefore
also $\fk b/(\fk b\cap w\fk b).$ This is
because if $\beta=w\al$ with $\beta, \al\in \Delta^+,$ then
$\gamma\beta\in\Delta^+,$ and  $\gamma\beta=w w^{-1}(\gamma)\al;$
since $w^{-1}(\gamma)\in \wti H,$ it stabilizes $(\fk b,\fk h),$ so
$w^{-1}(\gamma)\al\in\Delta^+.$
{
We will identify $ \fk b/(\fk b\cap w\fk b)$ with
\begin{equation}
  \label{eq:4.1.2c}
 Q_w:=\{\beta\in \Delta^+\ :\ \beta=-w\al \text{ with
 } \al\in \Delta^+\}.
\end{equation}
}
Write $Q_{w,\gamma}$ for the $\gamma$-orbits in $Q_\gamma.$ Then $\wti H$
acts on $\La^{\ell(w)}[\fk b/(\fk b\cap w\fk b)]$ by
$(-1)^{|Q_w|-|Q_{w,\gamma}|} e^{\rho-w\rho}.$ We write
$\ep(w):=(-1)^{|Q_w|},$ and $\ep_\gamma(w):=(-1)^{|Q_{w,\gamma}|}.$
\begin{lemma} For $h\in H_0,$
  \begin{equation*}
 \tr[h\gamma :
\sum (-1)^i\sideset{}{^i}{\bigwedge} \fk n]=\prod_{\beta\in\Delta^+_\gamma}
(1-e^\beta(h\gamma)  )
  \end{equation*}
\end{lemma}
\begin{proof}
We have
\begin{equation}
  \label{eq:4.1.2d}
  \tr[h\gamma:\sideset{}{^i}{\bigwedge}\fk n]=
\sum_{Q\subset\Delta^+,\ \gamma Q=Q}
(-1)^{i-|Q\cap\Delta^+_\gamma|} e^{<Q>}(h\gamma)
\end{equation}
where $Q\cap\Delta^+_\gamma$ are the $\gamma$-orbits in $Q,$
and $e^{<Q>}$ is the product of the $e^\beta$ with  $\beta\in
\Delta^+_\gamma\cap Q.$
{
The claim follows from the fact that
\begin{equation}
  \label{eq:4.1.3}
\sum_{Q\subset\Delta^+,\ \gamma Q=Q}
(-1)^{|Q\cap\Delta^+_\gamma|} e^{<Q>}(h\gamma)=
\prod_{\beta\in\Delta^+_\gamma}
(1-e^\beta(h\gamma)  ).
\end{equation}
}
\end{proof}

Let $(\chi,V_\chi)$ be an irreducible representation of $\wti H,$
and consider the Verma module
\begin{equation}\label{4.1.4}
M_\chi:=U(\fk k_c)\otimes_{U(\fk b)} V_\chi.
\end{equation}
This is a $(\fk k,\wti H)$ module.
\begin{proposition} Let $h\in H_0$ be such that $h\gamma$ is
regular. Then
\begin{equation*}
\tr M_\chi(h\gamma )=\sum_{\substack{Q=\sum m_\al \al,\\
\gamma(Q)=Q}}
\tr\chi(h\gamma)e^{-Q}(h\gamma)=\frac{\tr\chi(h\gamma)}
{\prod_{\beta\in \Delta^+_\gamma} (1-e^{-\beta}(h\gamma))}.
\end{equation*}
\end{proposition}

\begin{proof}As an $\widetilde H$-module $M_\chi$ is $S(\fk n^-)\otimes V.$
The weights of $S(\fk n^-)$ are all of the form $e^{-Q}$ with
$Q=\sum_{\al\in \Delta^+}m_\al\al.$ If the weight is not fixed by
$\gamma,$ it contributes zero to the trace. If it is, it
contributes the corresponding product of characters $e^\beta$
defined in (\ref{4.1.3}). The formula then follows in the usual
manner.
\end{proof}

\begin{theorem} Let $\pi_\mu$ be an irreducible
representation of $\widetilde K.$ Then
\begin{equation*}
\tr\pi_{\mu} (h\gamma)= \frac{\sum_{w\in W}\ep_\gamma
  (w)\tr e^{w\mu}(h\gamma)e^{w(\rho) - \rho}(h\gamma)}
{\prod_{\beta\in \Delta^+_\gamma} (1-e^{-\beta}(h\gamma))}.
\end{equation*}
\end{theorem}

\begin{proof}
{
 The trace can be computed as in the untwisted case by
  establishing a BGG type resolution of the representation
  $(\pi_\mu,V_\mu)$, \cite{BGG}.  Then the formula 
follows from proposition \ref{4.1}. 
}

{ We will sketch a different approach.  The cohomology is computed from the
complex
\begin{equation}
  \label{eq:4.1.6}
 \dots \longrightarrow V\otimes \sideset{}{^i}{\bigwedge}\fk
  n^*\overset{d^i}{\longrightarrow} V\otimes \sideset{}{^{i+1}}{\bigwedge}\fk
  n^*\longrightarrow\dots
\end{equation}
Then
\begin{equation}
  \label{eq:4.1.7}
  \begin{aligned}
&\sum (-1)^i\tr[\gamma ; H^i(\fk n,V)]
 = \sum (-1)^i\tr[h\gamma ; V\otimes \sideset{}{^i}{\bigwedge}\fk  n^*]=\\
&\tr\pi_\mu(h\gamma)\prod_{\beta\in\Delta^+_\gamma}(1-e^{-\beta}(h\gamma)),
  \end{aligned}
\end{equation}
 To prove the theorem it suffices to prove 
that as an $\wti H$-module,
\begin{equation}
  \label{eq:4.1.5}
  \Ho^i(\fk n,\pi_\mu) = \bigoplus_{l(w)=i} V_{w\mu}\otimes V_{w\rho - \rho}
\end{equation} }
 We follow \cite{GS}. Let
$\om^{-\al}$ be the basis dual to the root vectors, and
$\ep(\om^{-\al})$ the exterior wedge, and
$\iota(\om^{-\al})$ contraction  with $\om^{-\al}$. Then
$d^i=\partial +T,$ where
\begin{equation}
  \label{eq:4.1.8}
  \begin{aligned}
&\partial (f\otimes\om)=\sum \pi(E_{-\al})f\otimes \ep(\om^{-\al})\om,\\
&T(f\otimes\om)=\frac12\sum f\otimes \ep(\om^{-\al})E_{-\al}\om.
  \end{aligned}
\end{equation}
So $(d^i)^*=\partial^* +T^*$ is given by
\begin{equation}
  \label{eq:4.1.9}
  \begin{aligned}
&\partial^* (f\otimes\om)=-\sum\pi(E_{\al})f\otimes \iota(\om^{-\al})\om,\\
&T^*(f\otimes\om)=\frac12\sum f\otimes \iota(\om^{-\al})E_{\al}\om.
  \end{aligned}
\end{equation}
The basis $\{E_\al\}_{\al\in\pm\Delta^+}$ is in Weyl normal form; if we write
$[E_\al,E_\beta]=N_{\al,\beta} E_{\al+\beta},$ then
\begin{equation}
  \label{eq:4.1.10}
  \begin{aligned}
&E_\al\om^{-\beta}=N_{\al,\beta}\om^{-\al-\beta},\\
&E_{-\al}\om^{-\beta}=
\begin{cases}
N_{-\al,\beta}\om^{\al-\beta}\ &\text{ if } \beta-\al\in\Delta^+,\\
0 &\text{ otherwise.}
\end{cases}
\end{aligned}
\end{equation}
Then the cohomology is given by
harmonic forms, \ie forms annihilated by $d^i$ and $(d^i)^*.$ Section
6 of \cite{GS} then proves the result.

\end{proof}

\noindent {\bf Remark:} A similar  formula has also been obtained
by B. Kostant in \cite{Ko}.

\medskip
\subsection{} \label{4.2} We can construct a cover $\wwH$ such that the square
root of the character $e^{2\rho}$ (same as $\La^{top}(\fk n)$)
makes sense; namely take
\begin{equation}
  \label{eq:4.2.1}
\wwH:=\{\ (h,z)\in \wti H\times\bC^*\ |\ e^{2\rho}(h)=z^2 \ \}.
\end{equation}
The character $e^\rho$ is defined as the projection onto the
second component of $\wwH.$ Then $\mu\otimes e^\rho$ makes sense
on $\wwH$ and equals $-1$ on $(1,-1)$ and has differential $d\mu
+\rho.$ The formula in the proposition becomes
\begin{equation}\label{eq:4.2.2}
\tr\pi_{\mu}(h\gamma)= \frac{\sum_{w\in
W}\ep_\gamma (w)\tr\ e^{w(\mu+\rho)}(h\gamma)} {\prod_{\beta\in \Delta^+_\gamma}
(e^{\beta/2}(h\gamma)-e^{-\beta/2}(h\gamma))}.
\end{equation}
Separately, the numerator and denominator only make sense on
$\wwH$, but the ratio makes sense on $\wti H.$
\medskip

\subsection{}\label{4.3} The value at a singular $\gamma$ can be
computed by taking the limit $h\to 1.$  In this section we consider
the special case when
$\gamma=\tau$ is such that $\Ad\tau^2=Id$ on $K.$ We will assume (as we may)
that $\wti K$ is generated by $K$ and $\tau.$ Choose a pair $(\fk b, \fk h)$
which is stable under $\tau.$ Decompose
\[
\Delta^+=\Delta_{im}\cup
\Delta_{cx}=\Delta_c\cup\Delta_{nc}\cup\Delta_{cx}
\]
where $\Delta_{cx}$ are the roots such that $\tau\al\ne\al,$
$\Delta_{im}$ are the roots satisfying $\tau\al=\al,$ $\Delta_{nc}$
are the roots so that $\tau$ acts by $-1$ on the root vector and
$\Delta_c$ are the roots such that $\tau$ acts by 1 on the root
vectors. In other words, $e^\beta(\tau)=1$ for $\beta\in\Delta_c,$ and
$e^\beta(\tau)=-1$ for $\beta\in \Delta_{nc}.$ Then $\fk k(\tau)$ is
spanned by $\fk h^\tau$, $\{E_\al\}_{\al\in \pm\Delta_c}$, and $\{E_\al
+c_\tau E_{\tau\al}\}_{\al\in\pm\Delta_{cx}}.$  The restrictions of
the roots in $\Delta^+$ to $\fk h(\tau)$ form a positive system
$\Delta(\tau)^+$ of a reduced root system. The Weyl subgroup $W_\tau$
of $\wti K(\tau)$ corresponding to $(\fk b(\tau),\fk h(\tau))$ is
generated by the $s_\al$ for $\al\in\Delta_c$, $s_\al s_{\tau\al}$ for
$\al\in\Delta_{cx}$ such that $\langle \al,\tau\al\rangle=0,$ and
$s_{\al+\tau\al}$ if $\langle \al,\tau\al\rangle>0,$ or $
s_{\al-\tau\al}$ if $\langle \al,\tau\al\rangle<0.$ We write
$s_{\al,\tau}$ for these reflections. Let $W^+:=W/W_\tau.$ Each coset
has a unique representative $w$ such that if $\mu$ is dominant, then
$w\mu$ is dominant for $\Delta(\tau)^+.$  Then
\begin{theorem}[\cite{RS}, section 3]\label{4.3.3}
Suppose $F$ is $\tau$-stable finite dimensional with highest
weight $\mu$ satisfying $\tau\mu=\mu$ and $\langle
\mu,\cha\rangle \in2\bZ$ for all roots $\al.$ Then
\[
\tr F(\tau)\ne 0.
\]
\end{theorem}
\begin{proof}
A typical term in the numerator of (\ref{eq:4.2.2}) is of the form
\begin{equation}
  \label{eq:4.3.2}
  \ep_\gamma(xw)e^{xw(\mu+\rho(\tau))}e^{xw\rho-xw\rho(\tau)}
\end{equation}
When evaluated at $\tau$,
\begin{equation}
  \label{eq:4.3.3}
e^{xw\rho-xw\rho(\tau)}(\tau)=\ep_\gamma(xw)\ep(x).
\end{equation}
It follows that
\begin{equation}
  \label{eq:4.3.4}
\tr F(\tau)=\frac{\sum_{w\in W^+}\tr F_{w}(\tau)}
{\prod_{\beta\in \Delta_{nc}}
(e^{\beta/2}(\tau)-e^{-\beta/2}(\tau))}
\end{equation}
where $F_w$ is the representation of $\wti K(\tau)$ with highest
weight $w(\mu+\rho(\tau))-\rho(\tau).$ Because
of the condition $\langle \mu,\cha\rangle\in 2\bZ,$ $\tr
F_w(\tau)=\dim F_w.$ We conclude that
\begin{equation}\label{eq:4.3.5}
\tr F(\tau)=\frac{\sum_{w\in W^+}\dim
F_w}{2^{|\Delta_{nc}|}}.
\end{equation}

\section{\bf Lefschetz Numbers for Real Groups}\label{V}

Recall the notation in \ref{1.1}, $\wti
G=G\ltimes\{1,\tau\}$ for an element $\tau$ of finite order and
let  $\theta$ be a Cartan involution which commutes with $\tau.$ Then
$\tau$ stabilizes  the maximal compact subgroup $\wti K$. Its Lie
algebra $\fk k_0$ has a
$\tau$-stable Cartan subalgebra $\fk t_0$.

{If $\pi$ is an admissible representation of $\widetilde G,$ then
$\tau$ induces an automorphism of $\Ho^i(\fk g,K,\pi)$ (see
\cite{KnV} chapter II, section 6 for a definition of
$(\fk g,K)$-cohomology).
The Lefschetz number of $\tau$, $L(\tau,\pi)$, is by
definition the Euler characteristic of the trace of $\tau$ on
$\Ho^i(\fk g, K,\pi).$ More general, let $F$ be a finite dimensional
representation of $\wti G$ 
{
whose restriction to $G$ is  irreducible}.
Then we define the Lefschetz number of $\tau$ with respect to $F$ and $\pi$,
\begin{equation}
  \label{eq:5.1.3}
 L(\tau,F,\pi):=L(\tau,F^*\otimes \pi).
\end{equation}
 In this section we determine the Lefschetz numbers of
$\tau$. }

\subsection{Cohomology}\label{5.1} Recall that we do not
necessarily assume that $G$ is connected, rather that
$G=\mathbb G(\bR),$ the real points of a connected reductive group
$\mathbb G.$ {To determine the Lefschetz numbers we
follow \LABa\ and \BLS.} The module
$\bigwedge^*(\fk g/\fk k)\cong \bigwedge^*\fk s$ is a
representation of $\widetilde K.$
\medskip
Write
\begin{equation}
\chi_{\fk s}(k\tau ):=\chi(k\tau ;\sum (-1)^i\Lambda^i\fk
s)=\det[\one -\Ad(k\tau ):\fk s]. \label{5.1.1}
\end{equation}
The Lefschetz number is given by the formula
\begin{equation}
L(\tau,\pi):=\sum (-1)^i\tr(\tau,\Ho^i(\fk g,K,\pi))=\int_K
\chi_{\pi}(k\tau )\ov{\chi_{\fk s}(k\tau )}\ dk. \label{5.1.2}
\end{equation}

The following is well known, \cite{LAB1}, \cite{RS}, and\newline \cite{BW}.
\begin{proposition}\label{p:5.1} Assume $\pi$ is irreducible. Then
  $L(\tau,F,\pi)=0$ unless the restriction of $\pi$ to
$G$ is irreducible and the infinitesimal character of $\pi$ coincides
  with that of $F.$ It only depends on the  left coset of $\tau$ under $K.$
  In other words if $\beta=k\tau$ with $k\in K,$ then $L(\beta)=L(\tau).$
\end{proposition}

\end{proof}

\subsection{Standard Representations}\label{5.2}
In this section we construct the  irreducible representations with
nonzero Lefschetz  numbers.

{Let $P^0=M^0A^0N^0$ be a minimal parabolic
subgroup of $G$ such that $\fk a^0\subset \fk s_0,$ $M^0\subset K.$
Then $\tau(P^0)=\tau (M^0)\tau(A^0)\tau(N^0),$ and because $\tau$ commutes with
$\theta,$ $\tau\fk a^0\subset\fk s_0,$ and $\tau M^0\subset K.$ 

 Let $K_0$ be the connected component of 
the identity of $K.$ Since any two minimal parabolic subgroups
are conjugate by $K_0,$ there is
an element $k_0\in K_0$ such that $k_0\tau$ fixes $M^0,\ A^0$ and
$N^0.$ We will show that in fact there is a minimal parabolic subgroup
$P^0$ which is stable under $\tau.$

\begin{lemma}\label{l:5.2}
Let  $\tau_0$ be an isomorphism of $G$ which commutes with the Cartan
involution $\theta.$ Assume  $P^0$  is a minimal  $\tau_0$-stable
parabolic subgroup. Then the   map
  \begin{equation*}
    \al: K_0\times M^0\longrightarrow K_0,\qquad \al(k,m)=km\tau_0(k^{-1})
  \end{equation*}
is onto.
\end{lemma}
\begin{proof}
Because $K_0$ and $M^0$ are compact, the image of $\al$ is closed.
The set of $(k_0,m_0)$ for which
$d\al$ is \textit{not} onto is given by algebraic equations with real
coefficients. Thus its 
complement is either empty or Zarisky dense. It is enough to show it
is not empty. We show that there is a point $(k_0,m_0)$ such that
$d\al_{k_0,m_0}$ is onto. Take $k_0=Id.$ Then
\begin{equation*}
\exp Xm_0\exp\tau_0(-X)=
m_0\exp ( \Ad (m_0)^{-1}X)\exp -\tau_0(X),
\end{equation*}
and so the differential equals
\begin{eqnarray*}
  \label{eq:5.2.1a}
d\al_{k_0,m_0}(X,Y)&=&Y+ \Ad( m_0)^{-1} X-\tau_0(X)\\ &= &\Ad(
m_0)^{-1}[Y +X-\Ad(m_0)\tau_0(X).] 
\end{eqnarray*}
It is therefore enough to show that the map
\begin{equation}\label{eq:5.2.1b}
\al':X\mapsto X-\Ad( m_0)\tau_0(X)
\end{equation}
is onto the complement of $\fk m_0.$ Let $\Delta_0:=\Delta(\fk g,\fk a^0)$ be
the restricted roots, and $\{X_\al\}_{\al\in\Delta_0}$ a basis of root
vectors. Then the complement of $\fk m_0$ in $\fk k$ has as basis
$\{X_\al+\theta X_\al\}_{\al\in\Delta_0}.$ The claim that $\al'$ in
(\ref{eq:5.2.1b}) is onto follows from the
fact that we can choose $m_0$ such that $Id-\Ad( m_0)\tau_0$ has no fixed
points on the aforementioned basis.
\end{proof}

\medskip
\begin{proposition}[1]\label{c:5.2.1}
There exists a minimal parabolic $P^0=M^0A^0N^0$ satisfying
$\theta(\fk a^0)=\fk a^0, M^0\subset K$ and $\tau P^0=P^0.$
\end{proposition}
\begin{proof}
Since $P^0$ and $\tau P^0$ are both minimal parabolic subgroups, there
is $k\in K_0$ such that $\Ad( k)\tau P^0=P^0.$ Then $\tau':=\Ad
k\circ\tau$ does fix $P^0,$ so we can apply lemma \ref{l:5.2} whose
proof applies to any isomorphism $\tau_0$ not just ones of finite order. 
We conclude that  $k^{-1}=xm k\tau(x^{-1})k^{-1}$ or
$k=mx^{-1}\tau(x).$ Then $\Ad (x)P^0$ is $\tau-$stable. 
\end{proof} }

\medskip
We fix a minimal $\tau-$stable  parabolic subgroup $P^0$ with
properties as in corollary (1). Since $N^0$ is the radical of $P^0,$
we conclude $\tau N^0=N^0$ as well. 

Let now $P$ be a standard parabolic subgroup, \ie $P\supset P^0.$ Then
$\tau(P)$ is also a standard parabolic subgroup. If it has
decomposition $P=MAN,$ and it is $\tau$-stable, then $\tau M=M,\ \tau
A=A$ and $\tau N=N.$

\medskip
\begin{proposition}[2]\label{c:5.2.2}
If $P$ and $\tau P$ are conjugate under $K_0,$ then there is a conjugate
of $P$ which is $\tau$-stable.
\end{proposition}
\begin{proof}
The analogue of lemma \ref{l:5.2} holds by essentially the same
proof, and then the reasoning in the proof of corollary 1 in
\ref{c:5.2.1} applies.
\end{proof}

\subsection*{Remark} The arguments in lemma \ref{l:5.2} and
proposition (1) and (2) are adapted from \cite{ST1}.
{Note however that in proposition (2) we do need the assumption that
$P$ and $\tau P$ are conjugate under the \emph{ connected} group
$K_0$. To see that this is necessary, consider the  example  $G=GL(3)$ with 
$\tau(x)=\ ^tx^{-1},$ and $P$ a proper maximal parabolic subgroup.

Since $G=G(\bR)$ we do not assume in this paper that $K$ is connected
and so special care is needed in the arguments starting with the
notion of $\tau-$stable data in definition \ref{d:5.2}.\qed
 }

\bigskip
Let $X(P,\mathcal W,\nu)$ be a standard module for $G$, where $P=MAN$ is a
standard cuspidal parabolic subgroup (\ie $P^0\subset P$),
$\mathcal W$ is a tempered irreducible module for $M$ and $<Re\
\nu,\al> >0$ for all $\al\in\Delta(\fk n).$ We will always denote by
$L(P,\mathcal W,\nu)$ the irreducible Langlands quotient of $X(P,\C W,\nu).$

\medskip
Suppose an irreducible module $\pi$ of $\wti G$ has nonzero Lefschetz
number. By proposition \ref{5.1} its restriction to $G$ is
irreducible; it is of the form $L(P,\C W,\nu)$.
It must be isomorphic to $L(\tau P,\tau\C W,\tau
\nu).$ The modules $L(P,\mathcal W,\nu)$ and $L(\tau
P,\tau\mathcal W,\tau\nu)$ are equivalent if and only if there
exists $k\in K$ such that
\begin{equation}
k\tau M=M,\quad k\tau A=A,\quad k\tau N=N,\quad k\tau\mathcal W\cong
\mathcal W,\quad k\tau \nu=\nu. \label{5.2.1}
\end{equation}
\begin{definition}\label{d:5.2}
We say that the data $(P,\C W,\nu)$ are $\tau$-stable, if they satisfy
equation \ref{5.2.1}.   
\end{definition}

It follows that there is an
intertwining operator $a_{\tau}:\mathcal
W\longrightarrow \mathcal W,$ satisfying
\begin{equation}
a_{\tau}\pi_\C W(m)=\pi_\C W(k\tau(m))a_{\tau}. \label{5.2.2}
\end{equation}
Two choices of $a_\tau$ differ by a scalar multiple.
Then $a_{\tau}$ induces an intertwining operator \[A_{\tau}:
X(P,\mathcal W,\nu) \rightarrow X(\tau P,\tau\mathcal W,\tau\nu)\]
by the formula
\begin{equation}
  \label{eq:5.2.3}
  A_\tau(f)(g)=a_\tau f((k\tau )^{-1}(g)),\quad \text{ for } \quad
  f(gm)=\pi_\C W(m^{-1})f(g),
\end{equation}
satisfying $A_\tau^d=Id$ (because $a_\tau$ does).
{
Since $\nu$ is strictly dominant for $\Delta(\fk n)$, $\tau\nu$ is
strictly dominant for $\tau\Delta(\fk n).$ Thus  $L(P,\C W,\nu)$ is the unique
irreducible quotient of $X(P,\C W,\nu),$ and $L(\tau P,\tau\C
W,\tau\nu)$ is the unique irreducible quotient of $X(\tau P,\tau\C
W,\tau\nu).$  Thus $A_\tau$ induces a
nonzero operator $\C A_\tau$ from $L(P,\C W,\nu)$ to $L(\tau P,\tau\C
W,\tau\nu)$, and both are equivalent to $\pi.$   Any two such
operators are scalar multiples of each other.  We normalize $A_\tau$ so that
$\C A_\tau$ coincides with $\pi(k\tau).$ Thus the action of $G$ on
$X(P,\C W,\nu)$ extends to $\wti G,$ in such a way  that $\pi$ is its
unique irreducible quotient.}

\begin{theorem} Suppose the data $(P,\mathcal W,\nu)$ are
  $\tau$-stable so that $X(P,\C W,\nu)$ has an action of $\wti G.$
If $dim\ A\ge 1,$ then
\begin{equation*}
L(\tau, X(P,\mathcal W,\nu))=0.
\end{equation*}
\end{theorem}
\begin{proof} We can replace $\tau$ by $k\tau$ and apply formula
(\ref{5.1.2}). By Frobenius reciprocity, $\chi_\pi$ is supported on
$M.$ Thus we need to calculate
$\chi_{\fk s}(m\tau )$ with $m\in M\cap K.$ Because $m\tau$ has
fixed points (namely $\nu$) on $\fk a,$ formula (\ref{5.1.1}) equals
0.  See also \BLS.
\end{proof}
{
Thus if $\pi$ satisfies $L(\tau,\pi)\ne 0,$ then $\pi$ is a tempered
representation.}  
\subsection{Tempered Representations}\label{5.3} According to the
refinements of the Langlands classification due to Knapp and
Zuckermann  an irreducible tempered representation of $G$ is of
the form
\[\mathcal W=Ind_P^G[\mathcal W_0\otimes \bC_{i\nu}]\]
with $\mathcal W_0$ a \DS or \LDS\ representation, and $<\nu,\al> >
0$ for all $\al\in \Delta(\fk n)$ \cite{K-Z} .

\medskip
Suppose that the tempered representation ${\mathcal W}$ is
$\tau$-stable and that $H^*(\fk g,K,F^* \otimes {\mathcal W}) \not =
0$ for some finite dimensional $F.$
Because the infinitesimal character of $\mathcal W$ coincides
with that of a finite dimensional representation, $\nu=0$ and $\C
W$ is associated to a pair $(\fk b,\fk h)$ where $\fk b$ is a
$\theta$-stable Borel subalgebra and $\fk h\subset\fk b$ a
$\theta$-stable  \CSA. In particular, $\fk h$ is a fundamental
\CSA.

 Now let $H$ be the stabilizer in $G$ of  a $\theta$-stable pair $(\fk
b,\fk h)$; then $H =H_I\cdot H_R$ is a $\theta$-stable fundamental
\CSG.  Let $\fk h_I$ be the Lie algebra of $H_I:=H\cap K,$ and $\fk
h_R:=\fk h\cap \fk s.$ Since $G$ is the real points of a connected reductive
linear algebraic group, $H$ is abelian. Write $s=\dim(\fk n\cap\fk
k),$ $r=\dim(\fk n\cap\fk s)$ and let $\fk h_R$ be the
complexification of the Lie algebra of $H_R.$ Then
\begin{equation}
\mathcal W=\mathcal R^s_{\fk b}(\chi) \label{5.3.1}
\end{equation}
 where $\chi\in \widehat H_I$ is a
character such that $d\chi +\rho$ is dominant for $\fk b.$ (See
chapter V in \cite{KnV} for the definition of the functor
$\mathcal R^s_{\fk b}$.) We will write $\C R_{\wti G,\fk b}(\chi)$ or
$\C R_{G,\fk b}(\chi)$ or $\C R_{G_0,\fk b}(\chi)$ when we need to
emphasize whether the derived functor module is a $(\fk g, \wti K)$,
or $(\fk g, K)$ or $(\fk g, K_0)$ module.
{
The $(\fk g,K)$-module $\mathcal W$ is
$\tau$-stable if and only if there is $k\in K$ such that
$\gamma:=k\tau$ stabilizes the data $(\fk b,\fk h,\chi).$
By proposition \ref{p:5.1} when we compute Lefschetz numbers we can
replace $\tau$ by $\gamma=k\tau$ and thus the data $(\fk b,\fk
h,\chi)$ is $\gamma-$stable.}
\begin{theorem}\label{t:5.3} Let $\mathcal W=\mathcal R^s_{G,\fk b}(\chi)$, $\fk
  b=\fk h +\fk n$ be an irreducible tempered $\tau$-stable
  $(\fk g,K)-$module. Let $\gamma$
be as before. Assume that $F$ is an irreducible finite dimensional
$\gamma$-stable $(\fk g,K)-$module and that $F^{\fk n}=\chi$ as an
$\fk h$--module. The
Lefschetz number $L(\tau,F^*\otimes \mathcal {\mathcal W})$ equals
\begin{equation*}
(-1)^r\sum (-1)^{i}\ \tr (\ \gamma\ :\ \sideset{}{^i}{\bigwedge}
\fk h_R^* ).
\end{equation*}
It is zero if and only if $\gamma$ has fixed points in
$\fk h_R^*.$
\end{theorem}
\begin{proof} Denote by  $\bb C$ the trivial
representation of $H.$ Note that
\begin{equation}
\Ho^{i}(\fk g,K,F^*\otimes\C R^s_{\fk b}(\chi))\cong \Ext_{\fk
g,K}^{i}[F,\C R^s_{\fk b}(\chi)]. \label{5.3.3}
\end{equation}
 Corollary 5.121 in \KnV applies, and there is a first quadrant
spectral sequence
\begin{equation}
\label{eq:5.3.4}
 E_r^{p,q} \Longrightarrow
 \Ext_{\fk g,K}^{p+q-s}[F,\C R^s_{\fk b}(\chi)]
\end{equation}
with differential  of bidegree $(r,1-r)$ and with $E_2$ term
\[  E_2^{p,q}=\Ext_{\fk h,H_I}^{p}[H_{q}(\fk n,F),\chi]\]

\noindent
The $E_2$ term is nonzero only for $q=\dim\fk n.$ The conclusion
is
\begin{equation}
  \label{eq:5.3.5}
\Ho^{i}(\fk g,K,F^*\otimes\C R^s_{\fk b}(\chi))\cong \Ho^{i-r}[\fk
h, H_I,\bb C].
\end{equation}
In view of this, the Lefschetz number is
\begin{equation}
  \label{eq:5.3.6}
  \begin{aligned}
L(\tau,F^*\otimes\C W)&=\sum (-1)^i\tr (\gamma :\Ho^{i-r}[\fk h,
H_I,\bC])=\\
&=(-1)^r\sum (-1)^i\tr(\gamma, \sideset{}{^i}{\bigwedge} \fk h_R).
  \end{aligned}
\end{equation}
Finally if $\zeta_1,\dots ,\zeta_l$ are the eigenvalues of
$\gamma$ on $\Hom_{H_I}[\fk h_R,\bb C]$ (with multiplicities),
then
\begin{equation}
\sum (-1)^i tr[\ \gamma\ :\ \sideset{}{^i}{\bigwedge} \fk h_R^*
]=\prod(1-\zeta_j), \label{5.3.6}\end{equation} which is zero if and
only if one of the $\zeta_j$ equals 1.
\end{proof}

\medskip
\noindent\textbf{Example.}\ Consider the case $G=GL(n)$ with the
standard $\tau,$ transpose inverse. Then $\la=(\la_1,\dots ,
\la_n)$ is the highest weight of  a self dual finite dimensional
representation F if and only if $\la_i=-\la_{n+1-i}.$ There is
exactly one irreducible tempered $(\fk g,K)$--module $\mathcal W$
with nontrivial $(\fk g,K)$--cohomology with the same infinitesimal
character as $F$ \cite{sp}. If $n=2m,$ the Lefschetz number
$L(\tau, F^* \otimes {\mathcal W})$ is $(-1)^m2^m,$ if $n=2m+1,$
then it is equal to  $(-1)^{m+1}2^{m+1}.$

\bigskip
{ \section{\bf Lefschetz functions in the real case} \label{VI} By
\cite{Bou}, the distribution character of a $(\fk g,\wti K)$
module $\pi$ is given by a function $\Theta_\pi$ which is analytic
on the set of regular semisimple elements in $G^*.$ We want to
compute $\Theta_\pi$ on the regular elliptic set $H_{reg}\gamma$ for
the tempered representations $\mathcal W=\mathcal R^s_{\fk b}(\chi)$
considered in section \ref{5.3} This is known \cite{Bou}, but we
sketch a different treatment here based on derived functors.}

\medskip
\subsection{} Let $\pi$ be an admissible $(\fk g,\wti K)$ module.
The formal sum
\begin{equation}
\Theta_{\pi,\widetilde K}=\sum_{\mu\in \widehat{\widetilde K}}
m[V_\mu,\pi]V_\mu. \label{6.1.1}
\end{equation}
 is  a distribution on $\widetilde K$ in the sense that we
can replace $V_\mu$ by its character and evaluate on $\widetilde
K$-finite functions.

For a vector space $V,$ define a formal sum
\begin{equation}
e(V):=\sum (-1)^i\sideset{}{^i}{\bigwedge}V. \label{6.1.2}
\end{equation}
 If $V$ is a representation of some group, view
this sum in the corresponding Grothendieck group as a formal sum
of characters. 

\subsection{} \label{6.2} Recall $\C W=\C R^s_{G,\fk b}(\chi),$  $\fk
b=\fk h +\fk n,$ $s=\dim(\fk n\cap\fk k)$, $r=\dim(\fk n\cap \fk
s)$. We assume that $\C W$ is $\tau$-stable, and choose
$H$ and $\gamma$ as in (\ref{5.3.1}) 
{
\ie so that $(\fk b,\fk h,\chi)$
is $\gamma-$stable.} We lift $\chi$ to a double
cover as in section (\ref{4.2}) and tensor with $e^\rho.$ Since
the group $H_I$ is the stabilizer of $(\fk b,\fk h)$ in $K,$ it is
a normal subgroup of the stabilizer $H_K$ in $K$ of $(\fk b\cap
\fk k,\fk h\cap \fk k).$ So we have $H_0\subset H_I\subset H_K,$
where $H_0$ is the connected component of the identity in $H_I$, same
as the connected component of the identity in $H_K$. Let $\fk
b_j$ the conjugates of $\fk b$ under $H_K.$ Then the restriction
of $\C R_{G,\fk b}(\chi)$ to $G_0$ is a direct sum of $\C R_{G_0,\fk
b_j}^s(\chi_j)$ one for each $\fk b_j.$ Suppose $\gamma\in\wti
H_K$. Then $\gamma$ permutes the $\fk b_j,$ and in particular fixes
$\fk b.$  Let $T$ be the fixed points of $\gamma$ in $H_K.$ Then
$t\gamma$ permutes the $\fk b_j.$ If it does not fix any $\fk b_j,$
then $\tr\C W(t\gamma)=0.$ Thus we only need to compute $\tr\C
W(t\gamma)$ for $t\gamma$ which fix a $\fk b_j.$ Then $t\gamma$ is
conjugate to an element in $\wti H_I.$ Thus we only need to compute
the character for elements in $\wti H_I.$
{
\subsection*{Remark} The results and proofs in section \ref{2.1} are
for the case of an algebraically closed field, but they also hold for
compact connected groups. But since we do not assume that  $K$ is connected, 
it is not necessarily true that $H_K=H^\perp\cdot T.$\qed
}


{
Let $\Delta_\gamma^+$ be the roots in $\fk b$ which are not 1 on
$\gamma.$ Similarly $\Delta_\gamma(\fk s)^+$ is the set of roots in $\fk
b\cap\fk s$ which are not 1 on $\gamma.$ 
The  $\tilde{K}$ character $e(V)$ for  $V=\fk s$ equals
\begin{equation}
e(\fk s)(h\gamma)=e(\fk h _R)(\gamma)\prod_{\beta\in
\Delta_\gamma(\fk
s)^+}(1-e^{-\beta}(h\gamma))(1-e^{\beta}(h\gamma)). \label{6.1.3}
\end{equation}
}

Assume that $\rank \fk g =\rank \fk k.$ Then the spin representation
decomposes into a sum of two representations denoted $S^\pm.$ They
extend to $\wti K.$ The formula
\begin{eqnarray}
  \label{eq:6.2.1}
  (\tr S^+ -\tr S^-)(h\gamma)&= &\prod_{\beta\in \Delta_\gamma(\fk s)^+}
(e^{\beta/2}-e^{-\beta/2})(h\gamma)=\\
&=&(-1)^re^{-\rho(\fk n)+\rho(\fk n\cap \fk k)}e(\fk n\cap \fk s)(h\gamma).
\end{eqnarray}
holds. When $\rank\fk  k<\rank\fk g,$ the expression
\begin{equation}
  \label{eq:6.2.2}
e(S)^2= \prod_{\beta\in \Delta_\gamma(\fk s)^+}
(e^{\beta/2}-e^{-\beta/2})^2
\end{equation}
is a virtual character of $\wti K.$

\begin{proposition}[1] \label{p:6.2.1} Assume $\rank\fk k=\rank\fk g,$ and let
  $\C W=\C R^s_{\wti G,\fk b}(\chi)$
  with \linebreak $\fk b=\fk h +\fk n$ and $s=\dim(\fk n\cap\fk k)$. The formal
  combination $(S^+-S^-)\otimes \Theta_{\C W, \wti K}$ is a finite
  linear combination of irreducible $\widetilde K$
  representations. It equals the irreducible module with
  highest weight $\chi\otimes e^{\rho(\fk n)}.$
\end{proposition}
\begin{proof} We use the notation and results in \KnV chapter V.
Write $\chi^{\#}$ for the representation $\chi\otimes \wedge^r(\fk
n).$ If we denote by $W$ an arbitrary $\widetilde K$ module, then
\begin{equation}
\aligned
(-1)^r&\dim \Hom_{\widetilde K}(W,\mathcal R^s(\chi))=\\
=&\sum_{j=0}^s(-1)^j\sum_{n=0}^\infty \dim \Hom_{\widetilde H_I}
(\Ho_j(\fk n\cap\fk k,W),S^n(\fk n\cap \fk s)\otimes_{\bb
C}\chi^\#).
\endaligned
\label{eq:6.2.3}
\end{equation}
Tensoring $\mathcal R^s(\chi)$ with
$(S^+-S^-)$ in (\ref{eq:6.2.3}), and using the formula \[\sum S^n(\fk n\cap\fk
s)\cdot e(\fk n\cap\fk s)=1,\] we get
\begin{equation}\label{eq:6.2.4}
\aligned
&\dim\Hom_{\widetilde K}(W,\mathcal R^s(\chi)\otimes (S^+-S^-))=\\
&\ \ \ \ \ = \sum_{j=0}^s(-1)^j\dim\Hom_{\widetilde H_I} (\Ho_j(\fk n\cap\fk
k,W),\chi^{\#}\otimes e^{-\rho(\fk n)+\rho(\fk n\cap \fk k)}).
\endaligned
\end{equation}
Recall that $\chi^\#=\chi\cdot e^{2\rho(\fk n)},$ and the weights of
$H_j(\fk n\cap \fk k,W)$ are of the form $w(w_0\mu-\rho(\fk n\cap\fk
k))+\rho(\fk n\cap \fk k)$ with $w_0$ the longest element in $W_K.$ We get the equation
\begin{equation}
  \label{eq:6.2.5}
  w(w_0d\mu-\rho(\fk n\cap\fk k)) +\rho(\fk n\cap\fk k)=d\chi +\rho(\fk n)
  +\rho(\fk n\cap\fk k).
\end{equation}
Since $d\chi +\rho(\fk n)$ is dominant, it follows that $w=1,$ and
\[\mu\cdot e^{\rho(\fk n\cap \fk k)}=\chi\cdot e^{\rho(\fk n)}.\]
\end{proof}

Now assume that  $\rank \fk k< \rank\fk g.$
\begin{proposition}[2]\label{p:6.2.2}
The formal combination $e(S)^2\otimes\Theta_{\wti K}$ is a finite linear
 combination of irreducible representations of $\wti K.$ Assume
 $d\chi+\rho(\fk n)$ is very dominant. For a subset $B\subset \fk
 n\cap \fk s,$  let $<B>$ be the sum of roots in $B,$ and denote by
 $V(\chi\cdot e^{-<B>})$ the finite dimensional module of $\wti K$
 with this highest weight. Then
 \begin{equation*}
 e(S)^2\otimes \Theta_{\wti K}=  \sum_{i,|B|=i} (-1)^iV(\chi\cdot e^{-<B>}).
 \end{equation*}
\end{proposition}
\begin{proof}
We tensor $\C R^s(\chi)$ with $e(S)^2$ as in the proof of
\ref{p:6.2.1}:
\begin{equation}
  \label{eq:6.2.6a}
\aligned
&\dim\Hom_{\widetilde K}(W,\mathcal R^s(\chi)\otimes e(S)^2)=\\
&\ \ \ =\sum_{j=0}^s(-1)^j\dim\Hom_{\widetilde H_I} (\Ho_j(\fk n\cap\fk
k,W),\chi^{\#}\otimes e^{-\rho(\fk n)+\rho(\fk n\cap \fk k)}\otimes
e(\ovl{\fk n}\cap \fk s )).
\endaligned
\end{equation}
Because $d\chi +\rho(\fk n)$ is very dominant, $d\chi +\rho(\fk n)
-<B>$ is dominant, and the reasoning  in the proof of
proposition \ref{p:6.2.1} after (\ref{eq:6.2.5}) applies.
The claim follows.
\end{proof}


The distribution character of $\C W$ denoted $\Theta_\C W$ is given by
integration against a locally $L^1$ analytic function on the regular
set \cite{Bou}.

Let $\ep(w)$ for $w\in W_K$ be defined by
\begin{equation}
  \label{eq:6.2.6}
\prod_{\beta\in\Delta_\gamma(\fk s)^+} (e^{w\beta/2}-e^{-w\beta/2})=\ep(w)
 \prod_{\beta\in\Delta_\gamma(\fk s)^+} (e^{\beta/2}-e^{-\beta/2}).
\end{equation}

\begin{theorem} \label{t:6.2}Let $(\fk b,\fk h,\chi)$
be $\tau$-stable data for a tempered module $\C W =\C R_{\fk
b}^r(\chi).$ Then
\begin{equation*}
\Theta_{\C W}(h\gamma)= (-1)^{r} \frac{\sum_{w\in
W_K}\ep(w)w\chi(h\gamma)e^{w(\rho)-\rho}(h\gamma)}
{\prod_{\beta\in \Delta_\gamma} (1-e^{-\beta}(h\gamma))}
\end{equation*}
for any $h\gamma\in\wti H_{reg}$ stabilizing $\fk b.$
\end{theorem}

\begin{proof}  By results of Harish-Chandra,
  \cite{HC} sections 11 and 12, the distribution $\Theta_{\C W, \wti
    K}$ coincides with $\Theta_{\C W}$ when restricted to the regular
  set intersected with $\wti G_{reg}.$ The formula now follows from
  proposition (1) for $\rank\fk k=\rank\fk g.$ { If
    $\rank\fk k \not =\rank\fk g$ and $d\chi +\rho(\fk n)$ is very
  dominant, it follows
  from proposition (2) of \ref{6.2}  and in general by using
  translation functors as given in   chapter VII of \cite{KnV}.}
\end{proof}
As in section \ref{4.2} we can twist $\chi$ by $e^\rho$ and change
the definition of $\C R$ accordingly. We refer to \cite{KnV} for
details.  The formula is rewritten as
\begin{equation}
  \label{6.1.7}
 \Theta_\C W(h\gamma)=(-1)^{r}\frac{\sum_{w\in
     W_K}\ep(w)w\chi(h\gamma) e^{w\rho}(h\gamma)}
{\prod_{\beta\in\Delta_\gamma(\fk s)}(e^{\beta/2}(h\gamma)-e^{-\beta/2}(h\gamma))}
\end{equation}
for any $h\gamma\in\wti H_{I,reg}$ stabilizing $\fk b.$

\begin{corollary}\label{c:6.2} Let $F$ be finite dimensional $\tau$
  invariant irreducible representation and   $\gamma\in \wti H_K.$ Then
  \begin{equation*}
\Theta_F(h\gamma)=\sum \Theta_{\C W}(h\gamma)
  \end{equation*}
where the sum is over all the $\C W$ corresponding to the $\wti H_K$
conjugacy classes of $\gamma$-stable $(\fk b,\fk h,\chi).$
\end{corollary}

\bigskip

\subsection{}\label{6.3} Assume $\gamma$ is arbitrary semisimple.
Recall that $\gamma=\gamma_{ell}e^Y$ with $Y\in\fk g(\gamma)$
hyperbolic. We can conjugate $\gamma$ so that $Y$ is in a
$\theta$-stable \CSA \ $\fk h=\fk h_I+\fk h_R,$ in fact $Y\in\fk
h_R.$ Let $P=MN$ be the parabolic subgroup defined by $Y;$ the
roots $\Delta(\fk m)$ are the ones that are zero on $Y,$ the roots
$\Delta(\fk n)$ are the ones that are positive on $Y.$ Then
$G(\gamma)\subset M.$

If $f\in C_c^\infty(G\tau)$ and $\gamma$ normalizes $M,$ then
define
\begin{equation}
f^P(m\gamma )=\delta(m\gamma)^{1/2}\int_K\int_N f(km\gamma
nk^{-1})dn. \label{6.2.1}\end{equation} Then
\begin{equation}\label{6.2.2}
F_{f}^G(\gamma)=F_{f^P}^M(\gamma).
\end{equation}
is well defined.

\medskip
In \BLS, a function $f_{F}\in C_c^\infty(G\tau)$ is defined which
has the property that
\begin{enumerate}
\item $f_{F}(kxk^{-1})=f_F(x),$ \item $f^P_{F}=0$ for $P$ a real
parabolic whose conjugacy class
      is stable under $\tau$ (this means $P$ and $\tau(P)$ are conjugate under
      $G$),
\item $\Theta_\pi(f_F)=L(\tau,\pi\otimes F).$
\end{enumerate}
We refer to $f_F$ as the {\it Lefschetz function for $F$, $\tau$.}

\medskip
For the next results, keep in mind also that orbits of semisimple
elements are closed.

\begin{theorem} Let $f_F$ be a Lefschetz function for $F$,
  $\tau$. Suppose that $\gamma$ has nontrivial hyperbolic part. Then
\begin{equation*}
F^G_{f_F}(\gamma)=0.
\end{equation*}
\end{theorem}
\begin{proof} Apply formula (\ref{6.2.2}).
\end{proof}
\subsection{}\label{sec:6.3} In this section we compute the orbital
integrals on Lefschetz functions. For general results see
\cite{Re1} and \cite{Re2}.

\medskip
We use the conventions and notation of section \ref{III}. Assume
that $\gamma=k\tau$ with $k\in K$ is compact, and let $\fk t$ be a
{
fundamental} \CSA\  in $\fk g(\gamma).$ The centralizer of
$\fk t$ is a fundamental \CSA \ $\fk h$ of $\fk g.$  Fix $(\fk b,\fk h)$  a
$\gamma$-stable pair,  $H=H_I\cdot H_R$ the corresponding \CSG \ in
$G$ and $\wti H$ the \CSG \ in $\wti G.$  If $\fk b'$ is another
$\gamma$-stable Borel subgroup, let $\ep(\fk b'):=(-1)^{\dim [\fk
b'/(\fk b\cap\fk b')]}.$

\begin{theorem}\label{t:6.3} Suppose $f\in C_c^\infty(G^*)$ is such that $f^P=0$
  for all $P.$ Then there is a constant $c(\gamma)$ depending on the
  Haar measures on $G$ and $G(\gamma)$ such that
\begin{equation*}
F_{f}(\gamma)=c(\gamma)\sum|W(\chi)|^{-1}\ovl{\chi(\gamma)}\sum\ep(\fk
b')\Theta_{\C W(\fk b',\chi)}(f).
\end{equation*}
The first sum is over $\chi$ such that $d\chi$ is dominant for
$\fk b\cap\fk k$ and the second sum over $\fk b'\supset\fk
b\cap\fk k.$
\end{theorem}

\begin{proof}{ The idea of the proof originates in the
    work of Sally, Warner and Herb.}
    
Suppose $f$ is $C_c^\infty(G^*)$ supported on the regular
elliptic set. Then the function $F_f(h\gamma)$ is a well defined
function $\phi$ which is $C_c^\infty$ on $\wti H_{I,reg}.$ Its
Fourier transform is
\begin{equation}
\widehat{\phi}(\chi)=\int_{\wti H_I}{\chi(x)}\phi(x)\ dx.
\label{6.3.1}
\end{equation}
Assume for the moment that the support of $f$ is contained in $\Ad
G(T\gamma).$ Since $\phi$ is invariant (\ie we assume as we may by
averaging that $f(\Ad\gamma(x))=f(x)$ for $x\in G^*$),  Fourier
inversion gives
\begin{equation}
\phi(h\gamma)=\sum \tr\ovl{\chi(h\gamma)}\ \widehat{\phi}(\chi).
\label{eq:6.3.2}
\end{equation}
If the restriction of $\chi$ to $H$ is not irreducible, then
$\tr\chi(h\gamma)=0.$ In other words we may assume that $\chi$ is $\gamma$-stable so
1-dimensional and so we can suppress $\tr$ from the notation.

We now compute $\widehat\phi.$ On the one hand, because
$F_f(ht\gamma h^{-1})=F_f(t\gamma),$  we have
\begin{equation}
  \label{eq:6.3.3}
\widehat{\phi}(\chi)=\vol(H_I/T)\int_{T} |e(\fk h/\fk
t)(t\gamma) |\theta_\chi(t\gamma)F_f(t\gamma)\ dt
\end{equation}
where $e(\fk h/\fk t)$ is defined in (\ref{6.1.2}), and
\begin{equation}
  \label{eq:6.3.4}
\theta_\chi(t\gamma)=\sum_{w\in W}\ep(w)\ w\chi(t\gamma).
\end{equation}

On the other hand, for the tempered module corresponding to  $\fk
b'$ and $\chi-\rho,$ we can group the terms in the sum in
(\ref{eq:6.3.2}) according to the $\chi$ such that $d\chi$ is
dominant for $\fk b\cap\fk k.$ Fix a $\gamma$-stable $\fk b'$
which is dominant for $d\chi.$ Then
\begin{align}\label{eq:6.3.5}
&\Theta_{\C W(\fk b',\chi)}(f)=\ep(\fk b')\vol(H_I/T)
\int_{T} |e(\fk h/\fk t)(t\gamma)|\ \theta_\chi(t\gamma)F_f(t\gamma) + \\
\ &+\text{ (integrals of $f$ coming from Cartan subgroups of higher
real rank).} \notag
\end{align}

\medskip
\noindent So (\ref{eq:6.3.3}) is equal to the first term of
(\ref{eq:6.3.5}). By the continuity of the $F_f,$ the equality
holds for  all $C_c^\infty$ functions. In particular for a
cuspidal function $f_F$ the integrals coming from the more split
Cartan subgroups vanish and we get the claimed formula.
\end{proof}
\subsection{}\label{6.4} Fix a Haar measure on $G.$ There is a
canonical normalization of measures on the $G(\gamma),$ namely the
ones where $c(\gamma)=1.$ Equivalently, when $G(\gamma)$ is
{
elliptic} 
this measure is the one so that the formal dimension of
the discrete series with infinitesimal character equal to the one
of the trivial representation, is 1. These choices induce invariant
measures on the elliptic orbits. With this
normalization, the formulas in the previous sections simplify so
that there are no $c(\gamma).$ Furthermore note that  $r$
coincides with the number \[q(\gamma)=\frac12(\dim \fk g(\gamma)
-\dim \fk k(\gamma))\] associated to a real form of $G(\gamma)$ by
Kottwitz, so that \[(-1)^r=(-1)^{q(\gamma)}.\]

\begin{theorem}(1)\label{t:6.4.1}
Let $f_{F}$ be the Lefschetz function corresponding to a \linebreak
$\tau$-stable finite dimensional representation $F$ and suppose that $\gamma$
is elliptic. With the normalizations above,
$$
O_\gamma(f_{F})=(-1)^{q(\gamma)}e(\tau)\tr F^*(\gamma)
$$
where $e(\tau)=\sum_i (-1)^i\tr(\tau\ :\ \sideset{}{^i}\bigwedge\fk
h^*_R)$ as in proposition \ref{5.3}. 
{
In particular $O_\gamma(f_F)=0$ unless $\fk g(\tau)$ is equal rank.}

\end{theorem}
\begin{proof}
The formula follows from the above discussion and the fact that
the Lefschetz number is independent of the choice of $\gamma.$
\end{proof}

\begin{theorem}(2)\label{t:6.4.2}
 Fix an elliptic $\gamma.$
The stable combination of orbital integrals associated to $\gamma$
satisfies
\[
\sum (-1)^{q(\gamma')}O_{\gamma'}(f_F)=e(\tau)|\ker [\Ho^1(\Gamma,
I(\gamma))\longrightarrow \Ho^1(\Gamma,G)]| \tr F^*(\gamma).
\]
The sum on the left is over the stable conjugacy class of
$\gamma.$
\end{theorem}
\begin{proof}
If $\gamma$ and $\gamma'$ are elliptic stably conjugate
(definition \ref{7.3}), then $\tr F^*(\gamma)=\tr F^*(\gamma').$
{ The proof follows from the fact that 
\[|\ker [\Ho^1(\Gamma,I(\gamma))\longrightarrow \Ho^1(\Gamma,G)]|\] is the
number of stable conjugacy classes, (see proposition \ref{p:2.6}.}
\end{proof}

\subsection{}\label{6.5} Suppose that $\gamma$ stabilizes $(\fk b,\fk
h,\chi)$ with $\fk h$ fundamental as before. If $\gamma'$
stabilize the data as well, then $\gamma'\gamma^{-1}$ is in the
Cartan subgroup attached to $(\fk b,\fk h)$ which is abelian. Thus
$e(\gamma)=e(\gamma').$

\medskip

Now consider the restriction of $\C R_\fk b (\chi)$ to $G_0:$
\begin{equation}
  \label{eq:6.5.1}
  \C R_\fk b (\chi)=\sum \C R_{\fk b_i} (\chi_i).
\end{equation}
If none of the modules on the right are stabilized by $\tau,$
(with an element $k_0\tau$ with $k_0\in K_0$) then the Lefschetz
number is zero.

\medskip
{ So let $(\fk b_i,\fk h,\chi_i)$ be $\tau-$stable data of 
a summand in \ref{eq:6.5.1}. We can use it instead
of the original $(\fk b,\fk h,\chi).$ Thus we can assume that
$\gamma\in G_0\tau.$ Recalling the assumption that $\tau$ itself
is elliptic, corollary \ref{c:2.5} shows that $\gamma $ is
conjugate to $ h \tau $ and we can assume $h$ is in the Cartan
subgroup in $K_0.$  It follows that $e(\gamma)=e(\tau)$ because
$H$ is abelian.}

\section{\bf Lefschetz functions in the p-adic case}\label{VII}

{Recall the {\it twisted } orbital integral of a
function $f\in C_c^\infty (G\tau),$
\begin{equation*}
O_\gamma(f):=\int_{G(\gamma)\backslash G} f(g^{-1}\gamma g)\ dg.
\label{III.1}
\end{equation*}
In this section we compute the orbital integrals for Lefschetz
functions in the $p-$adic case. The results and techniques are
essentially in \KO. There are minor modifications due to the fact that $G$ is
reductive and possibly disconnected rather than semisimple.} The
definition of the Lefschetz function $f_\C L$  follows \cite{BLS}.

\subsection{}\label{7.1} In this section  ${\mathbb G}$ is a linear
algebraic reductive group, and $\tau$ an automorphism of finite
order, both defined over a nonarchimedean local field $\mathbf k$
of characteristic zero. Let $G:=\mathbb G(\bk).$


Now consider the building $\C B$ associated to ${\mathbb G.}$ Recall
that $G$ acts transitively on the chambers, and $\tau$
permutes them. Thus fix a chamber $C$ and let $\beta=b\ltimes\tau$
be in the stabilizer of $C.$  Denote by $\C F(\C B)$ the set of
facets of $\C B,$ and by $\C F(C)$ the facets of $C.$ These are
permuted by $\beta.$ Let $^0G$ be the intersection of the kernels
of the absolute values of all characters of $G.$ This is a normal
open compact subgroup of $G.$ Let $P_\sig$ be the stabilizer in
$^0G$ of the facet $\sig.$  Then $P_\sig$ is an open compact group
which we will call a parahoric subgroup. An element $x$ which
stabilizes the facet $\sig,$ permutes its vertices. Let $sgn_\sig (x)$
be the sign of this permutation.  Fix a  Haar measure $m$ of
$G.$ The Lefschetz function is
defined as
\begin{equation}\label{7.1.3}
f_{\C L}(x)=\sum_{\substack{\sig\in \C F(C)\\ \beta(\sig)=\sig}}
(-1)^{dim\sig} \frac{1}{m[P_\sig]}sign_\sig(x)\delta_{P\beta}(x).
\end{equation}

\subsection{}\label{7.2}

Let $\gamma=\delta'\tau=\delta\beta\in \wti G$ with $\delta\in G$  be a
fixed {\als element}. See section \ref{1.3} for the definition.
We want to evaluate $O_\gamma(f_\C L).$ Fix
$P$ a parahoric subgroup of $G$ corresponding to a
$\beta$-stable facet $\sig$ of $C.$ Write $\CP$ for the normalizer
of $P$ in $G$ and $X_P:=\ ^0G/P.$ Then $X_P$ is equivalent to the
set of facets of  type $P;$ the left action of $G$
corresponds to the standard action of $G$ on $\C B.$
Let
\begin{equation}
f_{P\beta}:=\frac{1}{m(P)}\delta_{P\beta}.
\label{7.2.1}\end{equation} Then $f_{P\beta}(g^{-1}\gamma g)\ne 0$
if and only if $g^{-1}\gamma g \beta^{-1}\in P,$ equivalently,
$g^{-1}\delta\beta(g)\in P.$ Thus
\begin{equation}
O_\gamma (f_{P\beta})=\frac{1}{m(P)}\  m\big[G(\gamma)\backslash\{G(\gamma)g\ :\
g^{-1}\delta \beta(g)\in P \}\big]. \label{eq:7.2.2}\end{equation} In
this formula $m$ refers to the quotient measure on
$G(\gamma)\backslash G.$

By possibly using a conjugate we may as well assume that
$\delta\in P,$ or else all integrals are zero anyway. If $g$
satisfies $g^{-1}\delta\beta(g)\in P,$ then so does $gn$ for any
$n\in P.$ Thus
\begin{equation}
O_\gamma(f)=\frac{1}{m(P)}
\sum_{\substack{g\in G(\gamma)\backslash G/P\\
g^{-1}\delta\beta(g)\in P}} m[G(\gamma)\backslash G(\gamma)gP].
\label{7.2.5}
\end{equation}

{ The group $G$ equals $\oG\cdot A$ where $A$ is the split component
of the center. Then $A\cong (\bF^\times)^r=GL(1,\bF)^r.$ The
lattice of coroots is $X^*(A)\cong \bZ^r.$ Then the automorphism
$\beta$ induces a linear isomorphism on this lattice, also denoted
$\beta$, satisfying $\beta^m=Id$ for some $m.$ There is a basis in
which $\beta$ is block diagonal with blocks corresponding to
irreducible factors of $t^m-1.$ Precisely, let
\begin{equation}
  \label{eq:7.2.6}
  t^s + b_{s-1}t^{s-1} + \dots + b_0
\end{equation}
be such a factor. On the basis of this block,
\begin{equation}
  \label{eq:7.2.7}
  \beta(a_0,\dots , a_{s-1})= (a_2,\dots , a_{s-2},a_0^{-b_{0}}\dots
  a_{s-1}^{-b_{s-1}}).
\end{equation}
Let $\oA:=A\cap \oG.$
Suppose $a\in A$ is such that $a^{-1}\beta(a)\in\oA.$ Using the block
decomposition of (\ref{eq:7.2.7}), we conclude that $a=a'x$ where
$\beta(a')=a'$ and $x\in\oA.$ } It follows
that we can replace $G$ by $^0G$ and $G(\gamma)$ by
$G_\#(\gamma):=^0G\cap G(\gamma).$ The condition
$g^{-1}\delta\beta(g)\in P$ is equivalent to
\begin{equation}
Ad(\gamma)(gPg^{-1})=gPg^{-1}.
\end{equation}
Let $R:=gPg^{-1}.$ Then
$$
G(\gamma)\backslash G(\gamma)gP\cong [G(\gamma)\cap R]\backslash
R\cong[G_\#(\gamma)\cap R]\backslash R.
$$
Then (\ref{eq:7.2.2}) becomes
\begin{equation}
O_\gamma(f_{\CP\beta})=\sum_{\sig\in \Gz(\gamma)\backslash
X_P(\gamma)}\frac{1}{m[G_\#(\gamma)_\sig]}.
\label{7.2.6}\end{equation} We conclude that
\begin{equation}
O_\gamma(f_\C L)=\sum_{\substack{\rho\in \Gz(\gamma)\backslash
X_P(\gamma)}} (-1)^{dim\rho}\frac{1}{m[G_\#(\gamma)_\rho]}.
\label{7.2.7}\end{equation}

\subsection{}\label{7.3}
Suppose $\bb H$ is a unimodular group acting in a {\it cell-wise}
fashion on a CW--complex (or more generally on a polysimplicial
complex) $\C T.$ Assume the following hold:
\begin{description}
\item[(i)] $\C T$ is contractible. \item[(ii)] $\C T$ is locally
compact. \item[(iii)] The stabilizer $\bb H_\sig$ of any cell
$\sig$ is an open compact subgroup of  $\bb H.$ \item[(iv)] Any
compact subgroup of $\bb H$ is contained in a $\bb H_\sig.$
\item[(v)] The number of cells are finite modulo the action of
$\bb H.$
\end{description}
Denote by $\Sigma$ the set of orbits of the cells. Let $m$ be an
invariant measure. Then write
\begin{equation}
\chi(m):= \sum_{\sig\in \Sigma}(-1)^{dim\sig}\frac{1}{m[\bb
H_\sig]} \label{7.3.1}\end{equation}
\begin{theorem}[\Se]
The measure $\mu=\chi(m)m$ is independent of $m$ and is an \EP
measure. If $\bb H$ is semisimple (or reductive but has a totally
anisotropic torus) then this measure is nonzero.
\end{theorem}
\subsection{}\label{7.4}  We show that  conditions (i)-(v) are
satisfied for $\bb H=G_\#(\gamma)$ and $\C T=\C B(\gamma).$
Items (i)-(iv) are straightforward.  For (v), suppose that $P$ is
stabilized by $\gamma.$ There is $x\in ^0G$ such that $P=x P_\sig
x^{-1}.$ It follows that  $x\gamma x^{-1}$ stabilizes $P_\sig,$
\ie  $x^{-1}\delta\beta(x)$ is in the normalizer $\C P_\sig$ of
$P_\sig.$ This is an open compact group. The orbit $O(\gamma)$ is
closed, so the set $\Gamma:=\{x^{-1}\delta\beta(x)\}$ is also
closed. Thus the intersection  $\Gamma\cap \C P_\sig$ is compact.
Thus there are $x_1,\dots ,x_n$ and a neighborhood $\C U\subset\C
P$ such that
\begin{equation}
  \label{eq:7.4.1}
  \Gamma \cap \C P_\sig=\bigcup_{i,u\subset\C U}
  ux_i^{-1}\delta\beta(x_i u^{-1})
\end{equation}
The claim follows.

\subsection{}\label{7.5} We say that $\gamma$ is elliptic if
$G(\gamma)$ contains a maximal anisotropic torus.
\begin{theorem}\label{t:7.5} The orbital integrals of $f_{\C L}$ are
$$
O_\gamma(f_{\C L})=\begin{cases} 1 &\text{ if } \gamma \text{ is elliptic,}\\
                                 0 &\text{ otherwise.}
\end{cases}
$$
\end{theorem}
\begin{proof}
The proof is the same as in \KO. The necessary modification were
discussed in sections \ref{7.1}-\ref{7.4}.
\end{proof}

\subsection{}\label{sec:7.6}
\begin{theorem}[\cite{BLS} section 8.2 and 8.4, and \cite{KO1}] Assume $G$ is
  simple. The  only irreducible unitary representations for which
  $\tr\pi (f_\C L)\ne 0,$ are the Trivial and the Steinberg
  representations. In these cases,
  \begin{equation*}
    \tr\pi(f_\C L)=
    \begin{cases}
      1 &\text{ if } \pi=Trivial,\\
     (-1)^{q(G)} &\text{ if } \pi=Steinberg,\\
    \end{cases}
  \end{equation*}
where $q(G)$ is the $\mathbf k$ rank of $G.$
\end{theorem}

\bigskip
\section{\bf The twisted trace formula}\label{VIII}

In this section we describe the trace formula and the effect of
plugging in a function which has local components as in sections
\ref{I}-\ref{VII}. The formulation of the simple version of the
trace formula we use can be found in \BLS \ and in \cite{KO1}. In
turn it is based on \Ar.

The assumptions on the group will be as in section \ref{I}.

\subsection{Generalities}\label{8.1}

Recall that $\mathbf K $ is a totally real number field. Let
$\chi$ be a  unitary character of $G(\bA)$ trivial on $G(\bK).$ We
assume that it satisfies $\chi\cong\chi\circ\tau$ so that it has
an extension to $\tilde{G}(\bA).$
If  $\bU $ is unipotent, we normalize the Haar measure $du_\bA$
so that
$meas(\bU(\bA)/\bU(\bK))=1$\newline where  $\bU(\bK)$ has the counting
measure. We fix a Haar measure $dk_\bA$ on the
maximal compact subgroup $K_\bA$ so that meas$(K_\bA )= 1$. Fix a
minimal parabolic subgroup $\bP_0=\bM_0\bU_0$ defined over $\bK$.
Fix a Haar measure $dm_\bA$ on $\bM_0(\bA).$ Then
\[f\mapsto
\int_{\bU(\bA)\bM_0(\bA)K_\bA}f(u_\bA m_\bA
k_\bA)m^{-2\rho_{p_0}}dk_\bA dm_\bA du_\bA \] defines a Haar
measure $dg_\bA$ on $\bG(\bA).$ We also fix a Haar measure
$dz_\bA$ on $\C Z(\bA)$.

Let $L^2(\mathbb G(\bA )/\mathbb G(\bK),\chi ) $ be the space of
square integrable functions on $\mathbb G(\bA)/\mathbb G(\bK)$ so
that $f(gz)=\chi(z) f(g)$ for $g \in \mathbb G(\bA)$ and
$z\in\mathcal Z(\bA)$.
 The group $\widetilde{ \mathbb G}(\bA)$ acts
unitarily on the space $L^2(\bG(\bA )/\bG(\bK),\chi)$.

{Let $\bA_G$ be the split component of the center of $\bG,$ and $\C
X^*(\bA_G)$ be its rational characters. Let $\fk a_{G,\bC}:=\Hom[\C
X(\bA_G),\bC],$ and $\fk a_G:=\Hom[\C X(\bA_G),\bR].$ The function
$H_G$ is defined as
\begin{equation}
  \label{eq:6.1.1p}
H_G:G(\bA)\longrightarrow\bR,\qquad H_G:a\mapsto (\chi\mapsto |\chi(a)|_\bA).
\end{equation}
Let $\obG(\bA)$ be the intersection of the absolute values of the
kernels of the rational characters of $\bG(\bA).$ The group
$(\bA_G)_\infty$ has a subgroup $A^+_G$ such that
$\bG(\bA)=\obG(\bA)\cdot A^+_G.$ }

The above discussion allows us to work with $L^2(\obG(\bA
)/\obG(\bK))$ instead of of $L^2(\mathbb G(\bA )/\mathbb G(\bK),\chi )
$. By abuse of notation we write \newline $L^2(\bG(\bA )/\bG(\bK))$ for
$L^2(\obG(\bA )/\obG(\bK))$.

\medskip
As reminder, the  goal of this article is  to show that there are irreducible
representations $\pi_\bA$  of ${\bG}(\bA)$ in
$L^2_{cusp}(\mathbb G(\bA )/\bG(\bK),\chi)$ so that $H^*(\fk g,K,
\pi_\bA \otimes F) \not = 0$ for some finite dimensional
representation $F$ such that $\pi_\bA\cong \pi_\bA\circ\theta$.
For this we will use the twisted Arthur trace formula on $\bG^*(\bA).$

\medskip
 We define a function $f_{\bA} = \prod
_{\nu} f_{\nu} $ on $\wti{\bG(\bA)}$ as follows. We fix a finite
dimensional $\theta$-stable representation $F$ of $\wti{\bG}(\bC
)$ with infinitesimal character $\la.$ For each infinite place
$\nu_\infty$ choose $f_{\nu_{\infty}}= f_{F} \in
C^{\infty}_c(\bG^*(\bR))$, the Lefschetz function  in section
\ref{6.2} attached to $F.$  For the finite number $S$ of places
where $\chi$ is not trivial on $\bG(\C O_\nu),$ choose $f_\nu$ to
have support in a small enough open set $K'_\nu$ on which
$\chi_\nu$ is trivial. We fix two finite places $\nu_0,\ \nu_1
\not \in S$ where we assume (as we may) that $\bK_{\nu_i}=
\bG(\C O_{\nu_i}).$ At these places we let $f_{\nu_i}=f_{\C L} $ be
the Lefschetz functions constructed in section \ref{sec:7.6}.  For
all other places let $h_{\nu }$ be the characteristic function of
a maximal compact subgroup $K_\nu \subset \bG(\bK_\nu).$ We
summarize the properties of the function $f_\bA$.
\begin{description}\label{properties}
\item[a] $\tr\pi (f_F) $ is $L(\tau,F,\pi)=e(\tau,\fk h_R)$ if
$\pi $ is a $\tau$ stable representation of the form $\C W=\C
R_{\fk
  b}(\chi)$ of $\wti{\bG}(\bR)$ with the same infinitesimal character
$\la$ (section \ref{5.3.1}). For other tempered representations,
$\tr\pi(f_F)=0.$  Furthermore $f_{F} $ is very cuspidal in the
sense of \LABa. \item[b] $\tr\pi_{\nu_i} (f_{\C L}) $ is equal to
1 if $\pi_{\nu_i} $ is the trivial or the Steinberg
representation. The trace  is zero on any other irreducible
representation. \item[c] Suppose $\gamma \in \bG^*(\bR)$. The
orbital integral
\[O_\gamma(f_\la) =
\int _{\bG(\gamma)(\bR)\backslash \bG(\bR)}f_\la (g^{-1}\gamma g)
dg
\] is 0 if $\gamma $ is regular semisimple but not elliptic.
\item[d] Suppose $\gamma \in \bG(\bk_{\nu_i}) $. The orbital
integral
$$
O_\gamma(\flef) = \int _{\bG(\gamma)(\bk_{\nu_i})\backslash
\bG(\bk_{\nu_i})} \flef (g^{-1}\gamma g) dg
$$
is $1$  if $\gamma$ is elliptic and zero otherwise, for $i=0,1$.
\end{description}

\medskip
\medskip
The twisted trace formula is  an identity of distributions
\[
\text{LHS}=\text{RHS}
\]
 on
$\obG^*(\bA)/\obG^*(\bK),$ where the right hand side is
parameterized by harmonic, i.e representation theoretic data,
whereas the left hand side is parameterized by geometric data.

\subsection{The harmonic side}\label{8.2}
  Following  the
notation in \Ar we write $R_{d,t} $ for the representations in the
discrete spectrum of the right regular representation of
$\widetilde{\bG}(\bA)$ on $L^2(\bG(\bA)/\bG(\bK))$ whose
infinitesimal character has length $t.$ Let $m_{disc}(\pi_{\bA} )$
be the multiplicity of a representation $\pi_{\bA} $ of
$\widetilde{\bG}(\bA)$ in the discrete spectrum of
$L^2(\bG(\bA)/\bG(\bK))$. We also write $R_{d,\la}$ for the
discrete spectrum  with infinitesimal character $\la.$

\medskip

\begin{proposition}
Let $f_{\bA} $ be as above. Then
\begin{equation*}
\mathrm{LHS}(f_{\bA}) = \sum_{\pi _{\bA} \in R_{d,\lambda }}
m_{disc}(\pi_{\bA }) \tr\pi_{\bA}(f_{\bA})
\end{equation*}
\end{proposition}
\begin{proof}
The function $f_\bA$ satisfies assumption a) and b) of 9.2 in
\BLS. Furthermore, since at the local place $\nu_0 $ the Lefschetz
function $f_\C L$ is a factor of $f_\bA$, the assumption c) of
9.5 in \BLS \ is satisfied. Thus by the formula 9.2 in \BLS \
$a_{disc}^L(\pi_{\bA} )= 0$ for $L\not = G$ and
$a_{disc}^G(\pi_{\bA}) = m_{disc}(\pi_\bA)$ \ (see proof of Corollary 7.3
in \cite{Ar}). So
\[ \text{LHS}(f_{\bA}) = \sum _{t \geq 0}
\sum_{\pi_{\bA} \in R_{d,t}} m_{disc}(\pi_{\bA} ) \tr
\pi_{\bA}(f_{\bA}). \] Taking into account that tr\
$\pi_{\infty}(f_{\lambda } )\neq 0$ only if the infinitesimal
character of  $\pi_{\infty}$ is equal to $\lambda,$ the sum over $t$
disappears, and $R_{d,t}$ is replaced by $R_{d,\la}.$
\end{proof}

\medskip
\subsection{}\label{8.3} Let $\bP=\bM \bN$ be a parabolic subgroup defined
over $\bK$ and $K_\bA$ a maximal compact group so that
$\bG(\bA)=\bP(\bA) K_\bA.$ Let $\bA_P$ be the split component of
the center of $\bM$ and $\C X^*(\bA_P)$ be its rational
characters. The complexified Lie algebra $\fk a_P$ of $A_P$ is
isomorphic to $\C X(A_P)\otimes_{\bZ}\bC.$  Let $\Delta(\fk
a_P,P)$ be the simple roots of $P$ and write $\rho_P$ for half
the sum of positive roots.

The group $(A_P)_\infty$ has a subgroup $A^+_P$ so that
$\bM(\bA)=\, ^0\bM(\bA)\cdot A^+_P,$ where $^0\bM(\bA)$ is the set of all
$m_\bA \in \bM(\bA)$ so that $|\chi(m_\bA)|_\bA=1$ for all
rational characters $\chi$ of $\C X^*(M)$. The function $H_\bP(\
)$ on $A^+_P$ is defined by the condition
\begin{equation}
  \label{eq:8.3.1}
  e^{\langle H_\bP(a) ,\chi\rangle} = |\chi(a)|.
\end{equation}
for all $\chi \in \C X(A_P).$

\medskip
Let $\C H= \C H _\infty \otimes \C H _f$ be the global Hecke
algebra.  If $X$ is an $\C H$--invariant space of automorphic
forms on $\bG(\bA)$, then the constant term $f_\bP$ of any $f\in
X$ along $\bP$ has an expression
\begin{equation}
  \label{eq:8.3.2}
  f_\bP(k_\bA m_\bA a \,  n_\bA )=\sum_i P_i(H_\bP(a)) a^{\mu_i+\rho_P}([\sum \phi_{i,j}(m_\bA)f_{i,j}(k_\bA)].
\end{equation}
The $P_i$ are polynomials, the $\phi_{i,j}$ are automorphic forms
of $^0 \bM(\bA)$ and the $f_{i,j}$ are $K_\bA$-finite functions.
The $\mu_i$ are distinct and the ones with nonzero contribution are
called the automorphic exponents of $f$ along $\bP$ and we call
the set of all $\mu_i$'s which appear as we vary $f$ over $X$ the
automorphic exponents of X along $\bP$.

\medskip

The \textit{local exponents} of $f$ at a place $\nu$ along $\bP$
are defined as follows: If $\nu$ is finite, then the Jacquet
module of the $\C H _\nu$ module  $(\C H_\nu*f_\nu)$ associated to
$\bN(\bK_\nu )$  is a finitely generated admissible
$\bM(\bK_\nu)$--module. The exponents at the place $\nu$ are the
absolute values of the characters of $\bA(\bK_\nu)$ that occur in
the Jacquet module.

\medskip
An automorphic function is called concentrated along $\bP$ if
$f_\bQ=0$ for any $\bQ$ which is not associate to $\bP.$

\medskip
\begin{theorem}[1]\label{t:8.3} (\KRS, 6.9)
Suppose $f$ that the automorphic form is concentrated along $\bP.$
Let $\mu$ be an
  automorphic exponent of $f$.
\begin{enumerate}
\item For any finite $\nu$ there is an exponent $\eta$ along $\bP$
so
  that \[Re(\mu)=\eta.\]
\item Suppose that $\nu$ is an infinite place. The generalized
eigenspace $H_0(\bn_\nu ,X)_\mu$ is non zero. Here $\bn_\nu$ is
the Lie algebra of $\bN(K_\nu)$.
\end{enumerate}
\end{theorem}

\medskip
\begin{proposition}[Rallis] \label{tempered} If an automorphic form is tempered at one
place, then it is cuspidal.
\end{proposition}
\begin{proof} We may as well assume that $f$ is concentrated along
$\bP.$ The condition for the local component to be tempered is
that the exponents should be of the form
$$
Re(\mu)=\rho_P +\sum x_\al\al,\qquad x_\al\ge 0,\ \al\in\Delta(\fk
a_P,\fk n).
$$
The condition for $f\in L^2(G(\bA)/G(\bK))$ is
$$
Re(\mu)=\rho_P -\sum y_\al\al^*,\qquad y_\al\ge 0,\
\al\in\Delta(\fk a_P,\fk n),
$$
where $\al^*$ is the dual basis to the simple roots. These two
conditions are incompatible unless the $\phi_{i,j}$ in
(\ref{eq:8.3.2}) are all zero.
\end{proof}

\medskip
\noindent\textbf{Remark:}\ In the nonadelic context, this
proposition is an earlier result of Wallach \Wa. The above adelic
version already appears in \cite{Cl}. \qed

\medskip

\begin{lemma}\label{onedimensional}
Suppose $\pi=\otimes\pi_\nu$ is such that $\pi_\nu$ is
1-dimensional for some $\nu.$ Then $\pi$ is 1-dimensional.
\end{lemma}
\begin{proof}
A 1-dimensional representation has a single exponent $\eta,$ and
this exponent satisfies $Re(\eta)=\rho_P.$  By theorem
\ref{t:8.3}, all automorphic exponents satisfy $Re(\mu)=\rho_P,$
and therefore for any place $\nu$ there is an exponent
$\nu_\nu$ satisfying $Re(\nu_\nu)=\rho_P.$ By  theorem 6.1 of
\cite{HM}, a unitary representation with this property has to be a
unitary character.
\end{proof}

\medskip
 The discrete spectrum of the regular representation  of $\wti{G}(\bA)$
 on \linebreak
$L^2(\bG(\bA)/{\bG}({\bK}))$ decomposes into a cuspidal part and a
residual part. Recall that $R_{d,\lambda} $ is  the set of
representations in the discrete spectrum with infinitesimal
character ${\lambda }$ and write $R_{c, \lambda} $ for the subset
of representations in the the cuspidal part.

\begin{theorem}[2]\label{t:8.3.2}
Every representation which contributes to $\mathrm{RHS}(f_{\bA})$
is either one dimensional or in the cuspidal spectrum.
\end{theorem}
\begin{proof} The Steinberg representation is tempered. So the
theorem follows from the previous propositions.
\end{proof}

\bigskip

\subsection{The geometric side}\label{8.4}
Recall that for  $\gamma =\{\gamma_\nu \}\in \bG^*(\bA)$
\begin{eqnarray*}
J_{{\bG}}(\gamma,f_\bA)&=& \int _{\bG(\gamma)^0(\bA)\backslash
\bG(\bA)} f_{\bA}(g^{-1}\gamma g) dg_{\bA}\\ & = & \prod_{\nu }
\int _{\bG(\gamma_{\nu })^0(\bk_{\nu})\backslash \bG(\bk_{\nu }) }
f_{\nu }(g^{-1}_{\nu}\gamma _{\nu } g) dg_{\nu}.
\end{eqnarray*}

\medskip
In the previous sections we have computed orbital integrals of the
form
\begin{equation}
  \label{eq:8.4.0}
  \int _{G(\gamma)\backslash G} f(g^{-1}\gamma g) dg.
\end{equation}
In what follows we will use
\begin{equation}
  \label{eq:8.4.0a}
  \int _{G(\gamma)^0\backslash G} f(g^{-1}\gamma g) dg.
\end{equation}
where $G(\gamma)^0$ is the connected component of the centralizer of
$\gamma.$ The relation between the two is a factor
$|G(\gamma)/G(\gamma)^0|.$

\medskip
 The results in 9.2 of \BLS \ combined with section 5 of \cite{KO1}
show that
\begin{equation}  \label{8.4.1}
\text{LHS}(f_{\bA}) = \sum _{\gamma \in({\bG^*} (\bK
))_{elliptic}} a^{G}(\gamma) J_{{\bG}}(\gamma,f_{\bA})
\end{equation}
where
\begin{equation}
  \label{eq:8.4.2}
a^G(\gamma)=\vol\bigg|\frac{\bG(\gamma)_0(\bA)}{\bG(\gamma)_0(\bK)}\bigg|\cdot
\bigg|\frac{\bG(\gamma)}{\bG(\gamma)_0}\bigg| .
\end{equation}
We note that the argument in \cite{KO1} about the geometric side
of the trace formula depends only on the fact that at one place
$v$, the component $f_v$ of $f_\bA$ is an Euler-Poincar\'e
function which in turn relies on results of Arthur for a connected
component  of a reductive group.

\medskip
\begin{lemma}[\cite{Cl}] \label{small}
Let $\C K\subset\G_\infty$ be a fixed compact set. There is a set
${{S}_1}$ of finite places with $\nu_0, \nu_1 \not \in {S_1}$ with
the following property. There is a choice of compact open
subgroups $K_\nu,\ \nu\in{S_1}$ so that if
$$
\gamma\in G(\bK)\cap \C K\prod_{\nu\notin S_1} G(\C O_\nu)
\prod_{\nu\in S_1}K_\nu
$$
then $\gamma$ is unipotent. The set $\C K\prod_{\nu\in
  S_1}K_\nu\prod_{\nu\notin S_1} G(\C O_\nu)$ can be chosen so
that it is $\tau$-stable.
\end{lemma}
\begin{proof} (included for completeness)
Choose any set $S_1$ of finite places that does not contain $\nu_0$ and
$\nu_1.$ Let $\rho: G\to GL(m)$ be a faithful representation. Let
\begin{equation}\label{8.4.3}
p(x,t):= \det(t-1+\rho(x))=t^m + a_{m-1}(x)t^{m-1} + \dots
+a_0(x).
\end{equation}
The $a_i$ are polynomials which extend to $G(\bA)$ and equal
\begin{equation}\label{8.4.4}
a_i(x)=a_{i,\infty}(x)
\prod_{\nu\notin S_1}a_{i,\nu}(x)
\prod_{\nu\in S_1}a_{i,\nu}(x).
\end{equation}
If all the $a_i(x)=0,$ then $x$ is unipotent. The first two factors
of the product are bounded. The last part can be made arbitrarily
small for $x_\nu\in K_\nu$ by making $K_\nu$ small enough. The
claim follows from the fact that for $x\in G(\bK),$ $|a_i(x)|_\bA$ is
either 1 or 0.
\end{proof}

\medskip

\begin{theorem} There is a choice of $f_\bA$ so that
\[
\sum_{\substack{\gamma \in (\wti{G} (\bK ))_{elipptic}\\
N(\gamma)=1}} a^{G}(\gamma) J_{\wti{G}}(\gamma ,f_{\bA})=
\sum_{\pi _{\bA} \in R_{d,\lambda }} m_{disc}(\pi_{\bA })
\tr\pi_{\bA}(f_{\bA})
\] As before, the sum is over (representatives of) conjugacy
classes. All representations contributing are either one
dimensional or in the cuspidal spectrum.
\end{theorem}
\begin{proof}
Recall that $f_\infty$ is a Lefschetz function, and has compact
support contained in a set $\tau\C K.$
Modify $f_{\bb A}$ so that $f_\nu$ is the delta function of
$K_\nu$ for $\nu\in S_1.$ Then apply lemma \ref{small} with $\C K$
as above to $\gamma^d$ to
conclude it must be the identity. Thus (\ref{8.4.1}) simplifies to
the formula in the proposition. See also \RS.
\end{proof}

\medskip

\bigskip

\section{\bf  A simplification of theorem \ref{8.4}}\label{IX}

In this section we combine the terms in $\text{RHS}(f_\bA)$ in
proposition  (\ref{8.4}) along stable conjugacy classes. The
references are \cite{KOSHE}, \cite{LAB2} and \cite{KO}. Most of
section is a summary of those results.

We consider in this section a connected reductive algebraic group
$\bG.$ This will be either the group considered in section I with
an automorphism  $\tau$ of finite order 
{
or the connected component $\bb
  I(\gamma):=\bb G(\gamma)^0$
of the centralizer of an elliptic element $\gamma$ in
$\bG^*$.}
Denote by $\bG_{der}$ the derived group and by
$\bG_{SC}$ its simply connected cover.

\bigskip

\subsection{}\label{9.1} Let $F$ be a global or local field.
  For
$\sig$ in the Galois group of $F $ and $g\in \bG(\bar{F})$
 we define a cocycle
by
\begin{equation}
  \label{eq:9.1.1}
 v_g(\sig):=g^{-1}\sig(g).
\end{equation}

Fix a semisimple element $\gamma=\delta\tau$ in $\bG^*(F)$ and let
  $\gamma'=g\gamma g^{-1}\in \bG^*(F)$ with $g\in \bG(\bar{F})$.
The cocycle $v_g$ takes values in $\bG(\gamma)(\bar{F})$ but not
necessarily in $\bb I(\gamma)(\bar{F}).$
\begin{definition}
 We say that two elements $\gamma,$ and $ \gamma'= g\gamma g^{-1} \in
 \bG^*(F)$ are \textit{stably   conjugate} if the cocycle $v_g$ of
 (\ref{eq:9.1.1}) takes values in   $\bb I(\gamma)(\bar{F})$ for all
 $\sigma $ in the Galois group of $F.$
\end{definition}

Conversely if $ v_g(\sig) $ is in $\bb I(\gamma)(\bar{F})$ for all
$\sigma $ in the Galois group of $F$ then $g\gamma g^{-1} \in
\bG(F).$

\medskip
If $\gamma $ is stably conjugate to $\gamma'=g\gamma g^{-1},$ then
the cocycle $v_g(\sig )$ in $H^1(F,\bb I(\gamma))$ belongs to
\[\mathcal{D}(\bb I/F)= \ker [H^1(F,\bb I(\gamma))
\rightarrow H^1(F,\bG)].\] See also 2.6.

\medskip \noindent
\subsection*{Remarks:}
\begin{enumerate}
\item If $\gamma $ and $\gamma'$ are stably conjugate then
$I(\gamma') $ is an inner twist of $I(\gamma)$.


\item Assume that $F$ is a number field, that G is a simply connected
 semisimple group. Then
 Kneser, Harder, Springer and Chernousov \cite{Ch} show that the Hasse principle
 holds, \ie
 \[ H^1(F,\bG(F)) \hookrightarrow \prod_v H^1(F,\bG(F_v)).\]
This implies that  $\gamma,\gamma'\in \bGti(F)$ are
conjugate by an element in $\bG(\bar{F})$ if and only if the
components in $\bGti(F_v)$ are conjugate by elements in
$\bG(\bar{F}_v)$.

\item Suppose that $F$ is a number field and that the Hasse
principle holds for the derived group $\bG_{SC}$. Let $\gamma \in
\bG(\bA)$. In 6.6 of \cite{KO}, R. Kottwitz defines an invariant
obs$(\gamma)$ which is trivial if and only if $\gamma $ conjugate
under $\bG(\bA)$ to an element in $\bG(F)$.
\end{enumerate}

\bigskip
 \subsection{A local example} Suppose $\bk$ is a local field and
$\bG=GL(n).$ We consider the automorphism $\tau(x):=\ ^tx^{-1}.$
An element $\gamma =\delta\tau \in
\bGti(\bk)$ is conjugate to $\tau$ if and only if
$\delta=g\tau(g^{-1})=gg^t,$ \ie it is a symmetric matrix.
Equivalently, we can think of the $\gamma'$s as quadratic forms and the
problem is then to classify them according to usual conjugacy under
$GL(n,\bk)$ and $GL(n,\ovl\bk).$
\begin{proposition}
An element   $\gamma = \delta \tau \in\bGti(\bk)$ is stably
conjugate to $\tau$ if and only if the determinant of $\delta$
is a square in $\bk^*.$ The stable conjugacy classes satisfying
$N(\gamma)=1$ are parametrized by $\bk^*/(\bk^*)^2.$
\end{proposition}
\begin{proof}
The centralizer of $\tau$ is the orthogonal group $O(n)$ which
has two connected components corresponding to $\det =\pm 1.$ Let
$\mathbb H$ be the diagonal Cartan subgroup which is both $\tau$
and $\gamma$ stable. The fact that a quadratic form over any field
$\bk$ can be diagonalized is equivalent to the fact that any
$\gamma$ is conjugate by $SL(n,\bk)$ to an element $\delta \tau$
with $\delta\in {\mathbb H }(\bk).$ It is clear that there is
$h\in {\mathbb H}(\ovl\bk)$ such that $\delta=hh^t.$ The element
$h$ can be chosen so that $\det h=\det\sig(h)$ for any $\sig\in
\Gamma$ precisely when $\det\delta\in (\bk^*)^2.$

The proof follows by recalling that $\gamma$ and $\gamma'$ viewed as
symmetric forms are
conjugate by an element in $\bG(\bk)$ if and only if the
discriminant and determinant of $\delta$ and $\delta'$ are equal
modulo squares in $\bk.$ 
\end{proof}

\subsection*{Remark}
By corollary \ref{2.1}, the condition $N(\gamma)=1$ in the
proposition is equivalent to the fact that
$\gamma$ is conjugate via $\bG(\bar{F})$ to the automorphism $\tau.$

\bigskip

\medskip
\subsection{}\label{9.6} Recall the formulas in section \ref{8.4}. Fix
Tamagawa measures on $\bG(\bA)$ and $\mathbb I(\gamma)(\bA).$ Then
the first factor in $a^G(\gamma)$  in (\ref{eq:8.4.2}) is the Tamagawa number of
$I(\gamma)(\bA)$ which we denote $\tau(\gamma).$ By \cite{KO1}, if
$\gamma$ is stably conjugate to $\gamma'$
\begin{equation}
  \label{9.6.1}
  \tau(\gamma)=\tau(\gamma').
\end{equation}

Thus \ref{8.4.1} becomes
\begin{theorem}
\begin{equation}
  \label{9.6.2}
\mathrm{RHS}(f_\bA)=\sum_{\gamma\in \Delta}
\tau(\gamma)\bigg|\frac{\bG(\gamma)_0}{\bG(\gamma)}\bigg|\sum_{\gamma'\in\C
  D(\bb I/\bK)}J_\gamma (f_\bA)
\end{equation}
where $\Delta$ is a set of representatives of  stable conjugacy
classes of elliptic semisimple elements $\gamma$ in $\bGti(\bK)$
satisfying $N(\gamma)=1$, and  $\C D(\bb I/\bK)$ parametrizes the
stable conjugacy class of $\gamma$ as in \ref{9.1}.
\end{theorem}
If $\bb I(\gamma)$ is simply connected, then $\tau(\gamma)=1.$

\section{\bf The main theorems}\label{sec:10}

We assume  $\bK=\bQ$.   We prove that $RHS(f_\bA)\ne 0$ and use
this to show that there exist $\tau-$invariant cuspidal automorphic
forms, and prove nonvanishing theorems for cuspidal cohomology. In
particular we illustrate these results in the case of $\bG=GL(n).$

\subsection{}\label{sec:10.1}
Since Tamagawa numbers are volumes, the coefficient of each
integral in \ref{9.6.2} is positive. We need a function $f_\bA$
such that the orbital integrals all have the same sign and at
least one is nonzero.

The orbital integrals have a product formula
\begin{equation}
  \label{eq:10.1.1}
J_\gamma(f_\bA)=[\prod_{\nu\ infinite} J_{\gamma_\nu}(f_\nu)]\cdot
[J_{\gamma_{\nu_0}}(f_{\nu_0})] \cdot
[J_{\gamma_{\nu_1}}(f_{\nu_1})] \cdot[\prod_{\nu\ finite, \nu \not
= \nu_0,\nu_1} J_{\gamma_\nu}(f_\nu)].
\end{equation}
By \ref{7.5} $J_{\gamma_{\nu_i}}(f_{\nu_i})=1 $ for i = 0, 1 and
$J_\gamma(f_\gamma)\ge 0$ for $\nu$ finite, $\nu\ne \nu_i.$ In
addition, if  $\nu $ is finite and  $\gamma=\tau,$
$J_\tau(f_\nu)>0$ . Recall that $F$ is a fixed  irreducible finite
dimensional representation of $\bG$ and that for each finite place
$\nu $ $f_\nu$ is a Lefschetz function $f_F$ and that by \ref{5.3}
$$
O_\gamma(f_{F})=(-1)^{q(\gamma)}e(\gamma)\tr F^*(\gamma)
$$
where \[e(\gamma)=\sum_i (-1)^i\tr(\gamma \ :\
\sideset{}{^i}\bigwedge\fk h^*_R)\] and
\[q(\gamma)=\frac12(\dim \fk g(\gamma) -\dim \fk k(\gamma))\] is
the number associated to a real form  $\bG(\gamma)$ by
Kottwitz. Therefore,
\begin{equation}
  \label{eq:10.1j}
J_\gamma=(-1)^{q(\gamma)}e(\gamma)\tr F^*(\gamma)\big|\frac{G(\gamma)}{G(\gamma)^0}\big|
\end{equation}

\medskip
We will restrict the support of the function $f_\bA$ at a finite
number of finite places such that the contribution of only one
$\gamma$ in theorem \ref{9.6.2} is nonzero.

Let $\Gamma=G(\bQ)\cap G_\infty K_{fin}$ where $K_{fin}$ is a
product of compact open subgroups as in lemma \ref{8.4}. This
choice depends on the function $f_\infty$ only.  A theorem of
Borel-Serre \cite{BSe}, section 3.8, states that {
$\Ho^1(\langle\tau\rangle,\Gamma)$} (notation \ref{2.6}) is finite
dimensional, \ie that the intersection of the set of elements
satisfying $N(\gamma)=1$ with $\Gamma$ breaks up into finitely
many orbits under $\Gamma$. Let $\tau_1=\tau,\dots , \tau_k$ be
representatives of these $\Gamma $-orbits.

\begin{lemma}\label{small enough}
There is an open compact subgroup $K_f=\prod K_\nu\subset K_{fin}$ with
$K_\nu=G(\C O_\nu)$ for all but finitely many places $S_1$ such that
$K_f\tau_i\cap K_f\tau_j=\emptyset$ for all $i\ne j.$
\end{lemma}
\begin{proof}
The elements $(\tau_i)_\nu$ are semisimple. So for each $\nu\in S_1$ replace
  $K_\nu$ by a smaller $K_\nu'$  so that the orbit of $(\tau_i)_\nu$ does not
  intersect $K_\nu'(\tau_1)_\nu.$
\end{proof}

Recall that the set $S$ was defined in section \ref{8.1} as the finite
set of finite places $\nu$  where the character $\chi$ is not trivial
on $\bG(\C O_\nu)$. The set $S_1$ is defined in lemma \ref{small}.

We also recall that according to theorem \ref{t:5.3}, $e(\tau)$ is
nonzero precisely when the centralizer of of $\tau$ in $\fk g(\bR)$ is
equal rank.

\begin{proposition}\label{p:10.1}  Let  $f_\bA  = \prod _{\nu} f_{\nu}
  $ be a function on $\wti{\bG(\bA)}$  satisfying the following properties.
\begin{enumerate}
 \item   $f_{\nu_{\infty}}= f_{F} \in C^{\infty}_c(\bG^*(\bR))$ is the
Lefschetz function defined  in section \ref{6.2}. \item For the
finite places $\nu_0,\ \nu_1 $
 $f_{\nu_i}=f_{\C L} $ is the Lefschetz
functions constructed in section \ref{sec:7.6}.
 \item For the places $\nu \in S \cup S_1$ let $f_{\nu }$ be the
characteristic function of a compact subgroup $K_\nu \subset
\bG(\bQ_\nu)$ which satisfies the assumptions in \ref{8.1},
\ref{small enough} and \ref{small}.
 \item  For all  places $\nu \not \in S \cup S_1$ let $f_{\nu }$ be the
characteristic function of a maximal compact subgroup $\tau_\nu K_\nu
\subset \bG(\bQ_\nu).$
\end{enumerate}

\noindent
 If $e(\tau)\ne 0,$ then  \[RHS(f_\bA)\ne 0.\]
\end{proposition}
\begin{proof}
The proof follows from the lemma and the discussion above.
\end{proof}


\medskip
In conclusion we have proved the following theorem.
\begin{theorem}\label{t:10.1.2} Let $\mathbb G$ be a connected
  reductive linear algebraic group
defined over   $\bQ$ and F a finite dimensional irreducible
representation of $\mathbb G({\bR}).$ If $\tr F(\tau)\ne 0$ and the
centralizer of $\tau$ in $\fk g (\bR)$ is equal rank, then
there exist cuspidal automorphic representations of ${\mathbb
G}({\bf A})$ stable under $\tau.$
\end{theorem}
\begin{proof} Let $f_\bA$ be the function in proposition
\ref{p:10.1}. Theorem \ref{8.4} and the previous proposition imply
that
\[  \sum_{\pi _{\bA} \in R_{d,\lambda }} m_{disc}(\pi_{\bA })
\tr\pi_{\bA}(f_{\bA}) \not = 0.\] If dim $F \
> \ 1$, then all the representations contributing to the sum are
in the cuspidal spectrum by the results in \ref{t:8.3.2}.

Suppose now that the representation $F$
is one dimensional. Denote the contribution  of the one dimensional
representations of $\bG(\bA)$ by ${\mathbf I}_\bA.$ Then
\[ {\mathbf I}_\bA(f_\bA) + \sum_{\pi _{\bA} \in R_{d,\lambda }}
m_{cusp}(\pi_{\bA }) \tr\pi_{\bA}(f_{\bA})=
a^G(\tau)J_\tau(f_\bA).
\]
{
We make the simplifying assumption that the center of
$\bG(\bA)$ is isomorphic to $(\bA^\times)^r$. A character of
$\bG(\bA)$ is determined by its values on the center. Let $\chi$ be a
character of $(\bA^\times)^r$ trivial on $(\bK^\times)^r$ is of the
form
\begin{equation}
  \label{eq:10.1.3}
  \chi(a_1,\dots ,a_r)=|a_1|^{s_1}\cdot\dots\cdot
  |a_r|^{s_r}\chi_1(a_1,\dots ,a_r)
\end{equation}
where $\chi_1$ is unitary. By the discussion in section \ref{8.1} we
can work with $\obG(\bA)$, the intersection of the kernels of the
absolute values of the characters of $\bG(\bA)$, we may as well assume
$s_1=\dots s_r=0.$ The character $\chi_1$ must be trivial at the
infinite places as well as $\nu_0,\ \nu_1.$ There are finitely many
places $v$ such that $(\chi_1)_v\ne triv.$ If such a place is not in
$S\cup S_1,$ then since $f_v=\one_{\C O_v},$ we have $\chi_1(f_v)=0.$
Thus $\bb I_\bA$ consists of finitely many characters.     

\medskip
Following the idea in \cite{BLS} we fix a finite place
$w \not \in S \cup S_1 \cup \{\nu_0,\nu_1\}.$ Choose a sequence of
compact open subgroups $K_w(i)$ (congruence subgroups 
$1+\varpi^i\bG(\C O)$) with characteristic functions $h_i^w$. Then
\[
J_{\tau_w}(h_i^w )=c_w q_w^{-i(\dim\bG-\dim\bG(\tau))}\longrightarrow 0
\text{ for } i\longrightarrow\infty,
\]
and $c_w\ne 0$ independent of  $i$. Similarly,
\[ 
\vol(K_w(i))=d_w q_w^{-i\dim\bG} \rightarrow 0 \mbox{ for }i \rightarrow
\infty .
\]  
A character $\chi_w$ satisfies $\chi_w(h_i^w)=0$ unless it is
trivial when restricted to $K_w(i).$ But there are at most
$q_w^{i\dim \C Z(\bG(F_w))^\tau}$ such characters. Thus if
\begin{equation}
  \label{eq:ineq}
  \dim\bG-\dim\bG(\tau)<\dim\bG-\dim\C Z(\bG)^\tau,
\end{equation}
$a^G(\tau)J_\tau(f_\bA(i))$ goes to zero strictly slower than $\bb
I_\bA(f_\bA(i)).$ It follows that there must be a nonzero
contribution from the cuspidal part. }
\end{proof}

\bigskip
\noindent{\bf Remark:} The number $\tr F^*(\theta)$ for $\tau=\theta$
an automorphism of order 2 is computed in section IV. In the case when
$\tau$ has order $d>2,$ we can see that there are infinitely many
finite dimensional representations satisfying $\tr F(\tau)\ne 0$ as
follows. It is enough to prove this for the case of finite dimensional
representations of the compact group $\wti K.$ Let $\wti K(\tau)$ be
the centralizer of $\tau.$ The Fourier expansion of the delta function
$\delta_\tau$ is
\begin{equation}
  \label{eq:10.1.2}
  \delta_\tau=\sum_{V\in\widehat{\wti K}(\tau)} \tr V(\tau)\dim V.
\end{equation}
Since $\delta_\tau$ is not smooth, there are infinitely many nonzero
terms in the right hand side. The claim that there are infinitely many
finite dimensional representations $F$ of $\wti K$ satisfying $\tr
F(\tau)\ne 0$ follows from the fact that the restrictions of the
representations of $\wti K$ span the Grothendieck group of $\wti K(\tau).$

\bigskip
\subsection{An example}\label{sec:10.2}
For ${\mathbb G}= GL(n)$ we consider the automorphism $\tau_c$
which is transpose inverse. Then $\tau_c(g_{\bA}) = g_{\bA}^{-1}$
for all $ g_{\bA} \in {\mathbb Z}(\bA)$. Therefore the set
$X(\tilde{\bG})_\bQ$ of $\bQ-$rational characters of $\wti{\bG}$ is
trivial,\linebreak  $a_{\tilde{G}} = Hom(X(\tilde{\bG})_\bQ,\bR )= 0$ and Arthur's
function $H_{\tilde{G}}$ equals zero. Thus in Arthur's notation
\[\tilde{\bG}(\bA)^1 =\tilde{\bG}(\bA).\]

\medskip
For $GL(n)$ the local measures and all the normalization factors
are as in \Cf \ page 261.
\medskip
We call a discrete series  representation D of GL(2,$\bR$) even if
\[ D(\left(%
\begin{array}{cc}
  -1 & 0 \\
  0 & -1\\
\end{array}%
\right)) = Id\] and odd otherwise. A tempered representation of
GL(n,$\bR$) induced from a maximal cuspidal parabolic subgroup
$P=MAN$ is even if it is induced from an even discrete series
representation  of every factor  of M. We call it odd if it is
induced from an odd representation on every factor of M.

Recall that  $\la=(\la_1,\dots , \la_n)$ is the highest weight of
a self dual finite dimensional representation F if and only if
$\la_i=-\la_{n+1-i}.$ The conditions of \ref{4.3.3} are satisfied
if
\begin{enumerate}
\item n is odd and all $\lambda_i$ are even,
 \item n is even and for all $i,j,\  \lambda_i = \lambda_j\ (mod\ 2)$
\end{enumerate}

\medskip
 For GL(n,$\bA $) all representations in the cuspidal
spectrum are tempered \cite{Sha}. There is exactly one irreducible
tempered $(\fk g,K)$--module $\mathcal W$ with nontrivial $(\fk
g,K)$--cohomology with the same infinitesimal character as $F$
\cite{sp} and
\begin{enumerate}
\item if n is odd then $\mathcal W$ is even
 \item if n is even then $\mathcal W$ is odd
\end{enumerate}
$\mathcal W$ is invariant under $\tau_c $ and has nontrivial
Lefschetz number.

\medskip

\begin{theorem}[1] \label{t:10.2.1}
There exist tempered cuspidal representations \linebreak  $\pi_\bA
= \prod \pi_\nu$ of GL(n,$\bA$) with the following properties:
\begin{enumerate}
 \item
$\pi_\bA$ is invariant under the Cartan involution $\tau_c$.
   \item
$\pi_\infty$  has an  integral nonsingular infinitesimal character
satisfying the conditions of \ref{4.3.3}.
 \item If n is even then $\pi_\infty $ an odd representation.
\item If n is odd then $\pi_\infty $ is an
even representation.
\end{enumerate}
\end{theorem}
\begin{proof}
This is essentially theorem \ref{t:10.1.2} combined with the results in
section \ref{4.3}.
\end{proof}
\medskip

For $\bG=GL(2m) $ we also consider the symplectic automorphism
$\tau_s$ with fix point set Sp(2m). The irreducible finite
dimensional representation $F$ with highest weight $(\la_1,\dots ,
\la_{2m})$ is invariant under $\tau_s$ if $\la_j =\la_1$ for $i =
2,\dots, m$ and $\la_j = -\la_1 $ for $i = m+1,\dots, 2m$. The
conditions of \ref{4.3.3} are satisfied if $\la_1 \in \bN$.

\begin{theorem}[2] \label{t:10.2.2}
There exist tempered cuspidal representations \linebreak $\pi_\bA
= \prod \pi_\nu$ of GL(2m,$\bA$) with the following properties:
 \begin{enumerate}
\item $\pi_\bA$ is invariant under the symplectic automorphism
$\tau_c$.
   \item
$\pi_\infty$  has an  integral nonsingular infinitesimal character
satisfying the conditions of \ref{4.3.3}.
\end{enumerate}
\end{theorem}

\medskip
The following is a generalization of the theorems (1) and (2)
using base change.
\begin{theorem}[3] \label{t:10.2.3}
Let $\bK/\bQ$ be an extension of $\bQ$ such that there is tower
\[ \bQ \subset \bK_1 \subset \bK_2 \subset \dots \subset \bK_r
=\bK\] of cyclic extensions of prime degree. There exist tempered
cuspidal representations $\Pi_{\bA_\bK}$ of
GL(n,$\bA_\bK$) with nontrivial cohomology.
\end{theorem}
\begin{proof}
Let $\bK/\bQ$ be a cyclic extension of prime degree, and $\pi_\bA$
a cuspidal representation of GL(n,$\bA$) with nontrivial
cohomology constructed in theorem (1). J. Arthur and L. Clozel proved
that there exists an
automorphic representation $\Pi_{\bA_\bK} $ of GL(n,$\bA_\bK$) lifting $\pi
$ (\cite{AC}, chap.3, theorem 4.2). This representation has a
Steinberg representation at 2 finite places (\cite{AC}, chap.1,
lemma 6.2) and is therefore also cuspidal. Furthermore this
representation has nontrivial cohomology.
\end{proof}

\medskip \noindent
{\bf Remark:} In the proof of theorem (2) no use is made of
property (2) of $\pi_\bA$ in theorem (1).

\bigskip
\subsection{} \label{10.3} In this section  we assume again that $\bG$ is defined over $\bQ$ and satisfies
of section \ref{1.1}. Consider the  locally symmetric space
\[S(K_f):= K_\infty
\ K_f \backslash {\bG}(\bA)/A_G \bG(\bQ)\] with $K_f$ \textit{small enough} as
in section \ref{sec:10.1}.  The DeRham cohomology
\[\Ho^*(S(K_f),F)\] with coefficients in the sheaf defined by  a
finite dimensional representation $F$ is isomorphic to
\[\Ho^*(\fk g,K_\infty,{\mathcal A}({\mathbb
G}(\bA )/A_G {\mathbb G}(\bQ))\otimes F)^{K_f}\] where
 ${\mathcal A}( {\mathbb G}(\bA )/{\mathbb
G}A_G(\bQ)) $ is the space of automorphic forms \cite{Fr1} and the
upper index denotes the invariants under $K_f$ . Denote by
\[{\mathcal A}_{cusp}( {\mathbb G}(\bA
)/A_G\bG(\bQ))\] the space of cusp forms. Then by \cite{Borel}
\begin{eqnarray*}
\lefteqn{\Ho^*(\fk g,K_\infty,{\mathcal A}_{cusp}({\mathbb G}(\bA
)/A_G\bG(\bQ))\otimes F))^{ K_f} }
\\ & & \hookrightarrow \Ho^*(\fk g,K_\infty,{\mathcal A}({\mathbb G}(\bQ)A_G \backslash {\mathbb G}(\bA ))\otimes F))^{
K_f}.\end{eqnarray*} The image is denoted by
$\Ho^*_{cusp}(S(K_f),F).$

\medskip
Let $F$ be a finite dimensional irreducible  representation of
$\bG$ which is invariant under an automorphism $\tau$ of $\bG$.
Then $\tau$ acts on $H^*(S(K_f),F)$.

\medskip
Let U be a maximal normal compact subgroup of $\bG(\bR)$.
 An involution $\tau$  is called  Cartan like if it defines an involution on
 $\bG(\bR)/U$ which is conjugate to a Cartan involution.
The assumptions of nonvanishing theorem \ref{t:10.1.2} are
satisfied for a Cartan like involution $\tau$ and the trivial
representation $F$. Thus theorem \ref{t:10.1.2} implies

\medskip
\begin{theorem}\label{t:10.3}
Suppose that $K_f$ satisfies the condition of proposition
\ref{p:10.1}. Let $\mathbb G$ be a connected reductive linear
algebraic group defined over $\bQ$.   Then
\[
\Ho^*_{cusp}(S(K_f),\bC) \not = 0.
\]

\end{theorem}

\medskip
 \noindent {\bf Remarks:}  In the equal rank case the
nonvanishing of the cuspidal cohomology was first proved by
L.Clozel and bounds on the cuspidal cohomology were obtained in
\cite{RS2} and in \cite{savin}. In these cases the ordinary trace
formula respectively the Euler-Poincar\'{e} characteristic was
used and no twisting by an automorphism was necessary.

For $S(K_f)$ is compact and $F$  nontrivial a nonvanishing theorem
was proved in \cite{RS} using geometric Lefschetz  numbers for
Cartan like involutions. For subgroups $\Gamma \subset SO(n,1)$
and $S(K_f)$ this theorem  was proved in \cite{RS3} also using
Lefschetz numbers.

 For semisimple $\bG$ and $S-$arithmetic groups it proved in
\cite{BLS} using $L^2$-Lefschetz numbers and a twisted trace
Arthur trace formula.

\ifx\undefined\bysame
\newcommand{\bysame}{\leavevmode\hbox to3em{\hrulefill}\,}
\fi

\end{document}

--------------030105030907010408060008
Content-Type: text/x-vcard; charset=utf-8;
 name="bes12.vcf"
Content-Transfer-Encoding: 7bit
Content-Disposition: attachment;
 filename="bes12.vcf"

begin:vcard
fn:Birgit Speh
n:Speh;Birgit
org:Cornell University;Department of Mathematics
adr;dom:;;Malott Hall 435;Ithaca;NY;14850
email;internet:speh@math.cornell.edu
tel;work:5-4620
version:2.1
end:vcard

--------------030105030907010408060008--